%% file: xu16a-20170730.tex
\documentclass[smallextended]{svjour3}

\usepackage{graphicx} 
\usepackage{subfigure}
\usepackage{amsmath}

\usepackage{amssymb}
\usepackage{color}
\usepackage{breqn}

\usepackage[ruled,vlined]{algorithm2e}

\usepackage{graphicx}
\usepackage{url}
\smartqed 

\input{macros.tex}

\begin{document}

\title{On the Convergence of Asynchronous Parallel Iteration with {Unbounded Delays}}

\author{Zhimin Peng \and
Yangyang Xu \and\\
Ming Yan \and
Wotao Yin
}

\institute{
Z. Peng and W. Yin \at Department of Mathematics, University of California, Los Angeles, CA 90095\\
\email{zhiminp@gmail.com / wotaoyin@math.ucla.edu}
\and 
Yangyang Xu \at Department of Mathematical Sciences, Rensselaer Polytechnic Institute, Troy, NY 12180\\
\email{xuy21@rpi.edu}
\and
Ming Yan \at Department of Computational Mathematics, Science and Engineering,
                         Department of Mathematics,
       Michigan State University,
       East Lansing, MI 48824\\
\email{yanm@math.msu.edu}
}




\maketitle

\begin{abstract}
Recent years have witnessed the surge of asynchronous parallel (async-parallel) iterative algorithms due to problems involving very large-scale data and a large number of decision variables. Because of asynchrony, the iterates are computed with outdated information, and the age of the outdated information, which we call \emph{delay}, is the number of times it has been updated since its creation. Almost all recent works prove convergence under the assumption of a finite maximum delay and set their stepsize parameters accordingly. However, the maximum delay is practically unknown.

This paper presents convergence analysis of an async-parallel method from a probabilistic viewpoint, and it allows for large {unbounded delays}. An explicit formula of stepsize that guarantees convergence is given depending on delays' statistics. With $p+1$ identical processors, we empirically measured that delays closely follow the Poisson distribution with parameter $p$, matching our theoretical model, and thus the stepsize can be set accordingly. Simulations on both convex and nonconvex optimization problems demonstrate the validness of our analysis and also show that the existing maximum-delay induced stepsize is too conservative, often slowing down the convergence of the algorithm.
\end{abstract}


\section{Introduction}
In the ``big data'' era, the size of the dataset and the number of decision variables involved in many areas such as health care, the Internet, economics, and engineering are becoming tremendously large~\cite{WH2014big}. It motivates the development of new computational approaches by efficiently utilizing modern multi-core computers or computing clusters.

In this paper, we consider the block-structured optimization problem
\begin{equation}\label{eq:main}
\Min_{\vx\in\RR^n} F(\vx)\equiv f(\vx_1,\ldots,\vx_m)+\sum_{i=1}^m r_i(\vx_i),
\end{equation}
where $\vx=(\vx_1,\ldots,\vx_m)$ is partitioned into $m$ disjoint blocks, $f$ has a Lipschitz continuous gradient (possibly nonconvex), and $r_i$'s are (possibly nondifferentiable) proper closed convex functions. Note that $r_i$'s can be extended-valued, and thus~\eqref{eq:main} can have block constraints $\vx_i\in X_i$ by incorporating the indicator function of $X_i$ in $r_i$ for all $i$.

Many applications can be formulated in the form of~\eqref{eq:main}, and they include classic machine learning problems: support vector machine (squared hinge loss and its dual formulation)~\cite{cortes1995-SVM}, LASSO~\cite{tibshirani1996-Lasso}, and logistic regression (linear or multilinear)~\cite{zhou2013tensor-reg}, and also subspace learning problems: sparse principal component analysis~\cite{zou2006sparse-PCA}, nonnegative matrix or tensor factorization~\cite{cichocki2009NMF-NTF}, just to name a few.

Toward solutions for these problems with extremely large-scale datasets and many variables, first-order methods and also stochastic methods become particularly popular because of their scalability to the problem size, such as FISTA~\cite{beck2009FISTA}, stochastic approximation~\cite{nemirovski2009robust}, randomized coordinate descent~\cite{nesterov2012RCD}, and their combinations~\cite{DangLan-SBMD,XuYin2015_block}. Recently, lots of efforts have been made to the parallelization of these methods, and in particular, asynchronous parallel (async-parallel) methods attract more attention (e.g.,~\cite{liu2014asynchronous,Peng_2015_AROCK}) over their synchronous counterparts partly due to the better speed-up performance.

This paper focuses on the async-parallel block coordinate update (async-BCU) method (see Algorithm \ref{alg:asyn_bcd}) for solving~\eqref{eq:main}. To the best of our knowledge, all works on async-BCU before 2013 consider a deterministic selection of blocks with an exception to~\cite{Strikwerda2002125}, and thus they require strong conditions (like a contraction) for convergence. Recent works, e.g.,~\cite{liu2014asynchronous,liu2015async-scd,Peng_2015_AROCK,hannah2016unbounded}, employ randomized block selection and significantly weaken the convergence requirement. However, all of them require bounded delays and/or are restricted to convex problems. The work~\cite{hannah2016unbounded} allows {unbounded delays} but requires convexity, and~\cite{davis2016asynchronous,cannelli2016asynchronous} do not assume convexity but require bounded delays. We consider {unbounded delays} and deal with nonconvex problems.

\subsection{Algorithm}
We describe the async-BCU method as follows. Assume there are $p+1$ processors, and the data and variable $\vx$ are accessible to all processors. We let all processors continuously and asynchronously update the variable $\vx$ in parallel. At each time $k$, one processor reads the variable $\vx$ as $\hat{\vx}^k$ from the global memory, randomly picks a block $i_k\in\{1,2,\cdots,m\}$, and renews $\vx_{i_k}$ by a prox-linear update while keeping all the other blocks unchanged. The pseudocode is summarized in Algorithm \ref{alg:asyn_bcd}, where the $\prox$ operator is defined in~\eqref{eq:prox}.

The algorithm first appeared in~\cite{liu2014asynchronous}, where the age of $\hat{\vx}^k$ relative to $\vx^k$,  {which we call the \emph{delay} of iteration $k$,} was assumed to be bounded by a certain integer $\tau$. For general convex problems, sublinear convergence was established, and for the strongly convex case, linear convergence was shown. However, its convergence for nonconvex problems and/or with {unbounded delays} was unknown. In addition, numerically, the stepsize is  {difficult to tune because it depends on $\tau$, which is unknown before the algorithm completes.}

\begin{algorithm}
\SetAlgoLined
\SetKwInOut{Input}{Input}\SetKwInOut{Output}{output}
\Input{Any point $\vx^0\in\RR^n$ in the global memory, maximum number of iterations $K$, stepsize $\eta>0$}
 \While{$k < K$, each and all processors asynchronously}{
  select $i_k$ from $[m]$ uniformly at random;\

  $\hat{\vx}^k\gets$ read $\vx$ from the global memory;\

  for all $i\in[m],$\
  \begin{equation}\label{main-update}
  \vx^{k+1}_i\gets\begin{cases}\prox_{\eta r_i}\left(\vx_i^k-\eta\nabla_i f(\hat{\vx}^k)\right),& \text{ if }i=i_k,\\
\vx_i^k,&\text{ otherwise};
  \end{cases}
  \end{equation}
  increase the global counter $k \leftarrow k+1$;\
 }
 \caption{Async-parallel block coordinate update}
 \label{alg:asyn_bcd}
\end{algorithm}

\subsection{Contributions} We summarize our contributions as follows.
\begin{itemize}
\item We analyze the convergence of Algorithm \ref{alg:asyn_bcd} and allow for large {unbounded delays} following a certain distribution. We require the delays to have  certain bounded expected quantities (e.g., expected delay, variance of delay). Our results  {are more general than} those requiring bounded delays such as~\cite{liu2014asynchronous,liu2015async-scd}.
\item Both nonconvex and convex problems are analyzed, and those problems include both smooth and nonsmooth functions. For nonconvex problems, we establish the global convergence in terms of first-order optimality conditions and show that any limit point of the iterates is a critical point almost surely. It appears to be the first result of an async-BCU method for general nonconvex problems and allowing {unbounded delays}. For weakly convex problems, we establish a sublinear convergence result, and for strongly convex problems, we show the linear convergence.
\item We show that if all $p+1$ processors run at the same speed, the delay follows the Poisson distribution with parameter $p$. In this case, all the relevant expected quantities  can be explicitly computed and are bounded. By setting appropriate stepsizes, we can reach a near-linear speedup if $p=o(\sqrt{m})$ for smooth cases and $p=o(\sqrt[4]{m})$ for nonsmooth cases.
\item When the delay follows the Poisson distribution, we can explicitly set the stepsize based on the delay expectation (which equals $p$). We simulate the async-BCU method on one convex problem: LASSO, and one nonconvex problem: the nonnegative matrix factorization. The results demonstrate that async-BCU performs consistently better with a stepsize set based on the expected delay than on the maximum delay. The number of processors is known while the maximum delay is not. Hence, the setting based on expected delay is practically more useful.
\end{itemize}
 {Our algorithm updates one (block) coordinate of $\vx$ in each step and  is sharply different from stochastic gradient methods that sample one function in each step to update all coordinates of $\vx$. While there are async-parallel algorithms in either classes and how to handle delays is important to both of their convergence, their basic lines of analysis are different with respect to how to absorb the delay-induced errors. The results of the two classes are in general not comparable. That said, for problems with certain proper structures, it is possible to apply both coordinate-wise update and stochastic sampling (e.g.,~\cite{recht2011hogwild,XuYin2015_block,MokhtaiKoppelRibeiro2016_class,davis2016asynchronous}),  and our results apply to the coordinate part.}

\subsection{Notation and assumptions}
Throughout the paper, bold lowercase letters $\vx, \vy,\ldots,$ are used for vectors. We denote $\vx_i$ as the $i$-th block of $\vx$ and $U_i$ as the $i$-th sampling matrix, i.e., $U_i \vx$ is a vector with $\vx_i$ as its $i$-th block and $\vzero$ for the remaining ones. $\EE_{i_k}$ denotes the expectation with respect to $i_k$ conditionally on all previous history, and $[m]=\{1,\ldots,m\}$.

We consider the Euclidean norm denoted by $\|\cdot\|$, but all our results can be directly extended to problems with general primal and dual norms in a Hilbert space.

The projection to a convex set $X$ is defined as
$$\textstyle \cP_X(\vy)=\argmin\limits_{\vx\in X} \|\vx-\vy\|^2,$$
and the proximal mapping of a convex function $h$ is defined as
\begin{equation}
\label{eq:prox}\textstyle\prox_h(\vy)=\argmin\limits_\vx h(\vx)+\tfrac{1}{2}\|\vx-\vy\|^2.
\end{equation}
\begin{definition}\label{def-critical}\textbf{\emph{(Critical point)}}
A point $\vx^*$ is a critical point of~\eqref{eq:main} if
$\vzero\in\nabla f(\vx^*)+\partial R(\vx^*),$
where $\partial R(\vx)$ denotes the subdifferential of $R$ at $\vx$ and
\begin{equation}\label{def-R}
R(\vx)=\sum_{i=1}^m r_i(\vx_i).
\end{equation}
\end{definition}

Throughout our analysis, we make the following three assumptions to problem~\eqref{eq:main} and Algorithm \ref{alg:asyn_bcd}. Other assumed conditions will be specified if needed.
\begin{assumption}\label{assump-sol}
The function $F$ is lower bounded. The problem~\eqref{eq:main} has at least one solution, and the solution set is denoted as $X^*$.
\end{assumption}

\begin{assumption}\label{assump-fun}
$\nabla f(\vx)$ is Lipschitz continuous with constant $L_f$, namely,
\begin{equation}\label{lip-f}
\|\nabla f(\vx)-\nabla f(\vy)\|\le L_f\|\vx-\vy\|,\,\forall \vx,~\vy.
\end{equation}
In addition, for each $i\in[m]$, fixing all block coordinates but the $i$-th one, $\nabla f(\vx)$  and $\nabla_i f(\vx)$ are Lipschitz continuous about $\vx_i$ with  $L_r$ and $L_c$, respectively, i.e., for any $\vx,~\vy$, and $i$,
\begin{align}
&\|\nabla f(\vx)-\nabla f(\vx+ U_i \vy)\|\le L_r\|\vy_i\|,\cr
&\|\nabla_i f(\vx)-\nabla_i f(\vx+ U_i \vy)\|\le L_c\|\vy_i\|.\label{lip-cd}
\end{align}
\end{assumption}
From~\eqref{lip-cd}, we have that for any $\vx,~\vy$, and $i$,
\begin{equation}\label{lip-ineq}
f(\vx+U_i\vy)\le f(\vx)+\langle\nabla_i f(\vx), \vy_i\rangle+\tfrac{L_c}{2}\|\vy_i\|^2.
\end{equation}
We denote $\kappa=\frac{L_r}{L_c}$ as the condition number.
\begin{assumption}\label{assump-read}
For each $k\ge 1$, the reading $\hat{\vx}^k$ is consistent and delayed by $j_k$, namely, $\hat{\vx}^k=\vx^{k-j_k}$. The delay $j_k$ follows an identical distribution as a random variable $\vj$
\begin{equation}\label{prob-delay}
\Prob(\vj=t)=q_t,\, t=0,1,2\ldots,
\end{equation}
and is independent of $i_k$. We let $$\textstyle c_k := \sum_{t=k}^\infty q_t,\qquad T:=\EE[\vj],\qquad S:=\EE[\vj^2].$$
\end{assumption}

\begin{remark}
Although the delay always satisfies $0\le j_k\le k$, the assumption in~\eqref{prob-delay} is without loss of generality if we make negative iterates and regard $\vx^k=\vx^0,\,\forall k<0$. For simplicity, we make the identical distribution assumption, which is the same as that in \emph{\cite{Strikwerda2002125}}. Our results can still hold for non-identical distribution; see the analysis for the smooth nonconvex case in the arXiv version of the paper.
\end{remark}

\section{Related works}
We briefly review block coordinate update (BCU) and async-parallel computing methods. 

The BCU method is closely related to the Gauss-Seidel method for solving linear equations, which can date back to 1823. In the literature of optimization, BCU method first appeared in~\cite{Hildreth-57} as the block coordinate descent method, or more precisely, block minimization (BM), for quadratic programming. The convergence of BM was established early for both convex and nonconvex problems, for example~\cite{luo1992convergence,Grippo-Sciandrone-00,Tseng-01}. However, in general, its convergence rate result was only shown for strongly convex problems (e.g.,~\cite{luo1992convergence}) until the recent work~\cite{hong2015iteration} that shows sublinear convergence for weakly convex cases.~\cite{tseng2009_CGD} proposed a new version of BCU methods, called coordinate gradient descent method, which mimics proximal gradient descent but only updates a block coordinate every time. The block coordinate gradient or block prox-linear update (BPU) becomes popular since~\cite{nesterov2012RCD} proposed to randomly select a block to update. The convergence rate of the randomized BPU is easier to show than the deterministic BPU. It was firstly established for convex smooth problems (both unconstrained and constrained) in~\cite{nesterov2012RCD} and then generalized to nonsmooth cases in~\cite{richtarik2014iteration,Lu_Xiao_rbcd_2015}. Recently,~\cite{DangLan-SBMD,XuYin2015_block} incorporated stochastic approximation into the BPU framework to deal with stochastic programming, and both established sublinear convergence for convex problems and also global convergence for nonconvex problems.

The async-parallel computing method (also called \emph{chaotic relaxation}) first appeared in~\cite{rosenfeld1969case} to solve linear equations arising in electrical network problems.~\cite{D-W1969chaotic-relax} first systematically analyzed (more general) asynchronous iterative
methods for solving linear systems. Assuming bounded delays, it gave a necessary and sufficient condition for convergence.~\cite{bertsekas1983distributed} proposed an asynchronous distributed iterative method for solving more general fixed-point problems and showed its convergence under a contraction assumption.~\cite{TB1990partially} weakened the contraction assumption to pseudo-nonexpansiveness but made more other assumptions.~\cite{FS2000asyn-review} made a thorough review of asynchronous methods before 2000. It summarized convergence results under nested sets and synchronous convergence conditions, which are satisfied by  P-contraction mappings and isotone mappings.

Since it was proposed in 1969, the async-parallel method has not attracted much attention until recent years when the size of data is increasing exponentially in many areas. Motivated by ``big data'' problems,~\cite{liu2014asynchronous,liu2015async-scd} proposed the async-parallel stochastic coordinate descent method (i.e., Algorithm \ref{alg:asyn_bcd}) for solving problems in the form of~\eqref{eq:main}. Their analysis focuses on convex problems and assumes bounded delays. Specifically, they established sublinear convergence for weakly convex problems and linear convergence for strongly convex problems. In addition, near-linear speed up was achieved if $\tau=o(\sqrt{m})$ for unconstrained smooth convex problems and $\tau=o(\sqrt[4]{m})$ for constrained smooth or nonsmooth cases. For nonconvex problems,~\cite{davis2016asynchronous} introduced an async-parallel coordinate descent method, whose convergence was established under iterate boundedness assumptions and appropriate stepsizes.

\section{Convergence results for the smooth case} Throughout this section, let $r_i=0,\,\forall i$, i.e., we consider the smooth optimization problem
\begin{equation}\label{eq:main-sm}
\textstyle\Min\limits_{\vx\in\RR^n} f(\vx).
\end{equation}
The general (possibly nonsmooth) case will be analyzed in the next section. The results for nonsmooth problems of course also hold for smooth ones. However, the smooth case requires weaker conditions for convergence than those required by the nonsmooth case, and their analysis techniques are different. Hence, we consider the two cases separately.

\subsection{Convergence for the nonconvex case} In this subsection, we establish a subsequence convergence result for the general (possibly nonconvex) case. We begin with some technical lemmas. The first lemma deals with certain infinite sums that will appear later in our analysis. 
\begin{lemma}\label{lem:thm-cvg-lem1}
For any $k$ and $t\le k$, let
\begin{subequations}\label{def-para}
\begin{align}
&\textstyle\gamma_k=\frac{\eta^2 L_r}{2m\sqrt{m}}\sum_{d=1}^{k-1} (c_{k-d}-c_k) c_d+\frac{\eta}{2m}c_k+\frac{\eta^2 L_c}{2m}c_k, \label{def-para_a}\\
&\textstyle\beta_k=\left(\frac{\eta}{m}-\frac{\eta^2 L_c}{2m}\right)q_0-\frac{\eta}{2m}c_k \textnormal{ for }k\geq 1, ~~\left(\textnormal{and }\beta_0=0\right),\\
&\textstyle C_{t,k}=\left(\frac{\eta}{m}-\frac{\eta^2 L_c}{2m}\right)q_t-\frac{\eta^2 L_r}{2m\sqrt{m}}\left(tq_t+\sum_{d=1}^{t}(c_d-c_{k})q_{t-d}\right).
\end{align}
\end{subequations}
Then
\begin{align}
&\textstyle\sum_{k=0}^\infty \gamma_k\le {\tfrac{\eta^2L_r}{2m\sqrt{m}}}T^2+\left(\frac{\eta}{2m}+\frac{\eta^2 L_c}{2m}\right)(1+T),\label{eq:bd-gammak}\\
&\textstyle\beta_k+\sum_{t=k+1}^\infty C_{t-k,t}\ge \frac{\eta}{2m}-\frac{\eta^2 L_c}{2m}-\frac{\eta^2 L_rT}{m\sqrt{m}},\,\forall k.\label{eq:bd-beta-Ck}
\end{align}
\end{lemma}
\begin{proof} To bound $\sum_{k=0}^\infty \gamma_k$, we bound the first term $\sum_{d=1}^{k-1} (c_{k-d}-c_k) c_d$ in~\eqref{def-para_a}.
Specifically,
\begin{align*}
\sum_{k=0}^\infty\sum_{d=1}^{k-1} (c_{k-d}-c_k) c_d\leq  \sum_{k=0}^\infty\sum_{d=1}^{k-1} c_{k-d} c_d= \sum_{d=1}^\infty\sum_{k=d+1}^{\infty} c_{k-d} c_d=T^2,
\end{align*}
where the last equality holds since
$T:=\EE[\vj]=\sum_{t=1}^\infty t q_t=\sum_{t=1}^\infty \sum_{d=1}^t q_t=\sum_{d=1}^\infty \sum_{t=d}^\infty q_t=\sum_{d=1}^\infty c_d.
$
We obtain~\eqref{eq:bd-gammak} by combining these two equations.

To prove~\eqref{eq:bd-beta-Ck}, we will use
\begin{equation}\label{ineq-cd}
\textstyle \sum\limits_{t=1}^\infty\sum\limits_{d=1}^{t}(c_d-c_{k+t})q_{t-d}\le\sum\limits_{t=1}^\infty\sum\limits_{d=1}^{t}c_d q_{t-d}=\sum\limits_{d=1}^\infty\sum\limits_{t=d}^\infty c_d q_{t-d}=\sum\limits_{d=1}^\infty c_d=T.
\end{equation}
The above inequality yields
\begin{align*}
 &\textstyle \beta_k+\sum_{t=k+1}^\infty C_{t-k,t}=\beta_k+\sum_{t=1}^\infty C_{t,t+k}\\
=  &  \textstyle \big(\frac{\eta}{m}-\frac{\eta^2 L_c}{2m}\big)q_0-\frac{\eta }{2m}c_k\\
   & \textstyle +\sum_{t=1}^\infty\Big(\big(\frac{\eta}{m}-\frac{\eta^2 L_c}{2m}\big)q_t-\frac{\eta^2 L_r}{2m\sqrt{m}}tq_t-\frac{\eta^2 L_r}{2m\sqrt{m}}\sum_{d=1}^{t}(c_d-c_{k+t})q_{t-d}\Big)\\
\stackrel{(\ref{ineq-cd})}\ge & \textstyle  \frac{\eta}{m}-\frac{\eta^2 L_c}{2m}-\frac{\eta }{2m}c_k-\frac{\eta^2 L_rT}{m\sqrt{m}}\ge \frac{\eta}{2m}-\frac{\eta^2 L_c}{2m}-\frac{\eta^2 L_rT}{m\sqrt{m}},
\end{align*}
where the last inequality follows from $c_k\le 1$.
\qed
\end{proof}
The second lemma below bounds the cross term that appears in our analysis.
\begin{lemma}[Cross term bound]\label{lem:thm-cvg-lem2}
For any $k > 1$ and $t\le k$, it holds that
{\begin{align}\label{eq:thm-cvg-t1-eq}
    &\textstyle \sum_{t=1}^{k-1}q_t\EE \Big[-\langle \nabla f(\vx^k)-\nabla f(\vx^{k-t}), \nabla f(\vx^{k-t})\rangle\Big]\\
\le &\textstyle \frac{\eta L_r}{2\sqrt{m}}\sum_{t=1}^{k-1}\left(tq_t+\sum\limits_{d=1}^t (c_{d}-c_k)q_{t-d}\right)\EE \|\nabla f(\vx^{k-t})\|^2\cr
    &+\textstyle \frac{\eta L_r}{2\sqrt{m}}\sum_{d=1}^{k-1} (c_{k-d}-c_k) c_d\|\nabla f(\vx^0)\|^2.\nonumber
\end{align}}
\end{lemma}
\begin{proof}
Define $\Delta^d:=\nabla f(\vx^d) - \nabla f(\vx^{d+1})$. Applying the Cauchy-Schwarz inequality with $\nabla f(\vx^{k-t}) - \nabla f(\vx^{k}) = \sum_{d=k-t}^{k-1}\Delta^d$ yields
$$\textstyle -\langle \nabla f(\vx^k)-\nabla f(\vx^{k-t}), \nabla f(\vx^{k-t})\rangle\le \sum_{d=k-t}^{k-1} \|\Delta^d\|\cdot \|\nabla f(\vx^{k-t})\|.$$
Since $\|\Delta^d\|\le L_r\|\vx^{d+1}-\vx^d\| = \eta L_r\|\nabla_{i_d}f(\hvx^{d})\|,$ by applying Young's inequality, we get
\begin{align}\label{eq:lem:thm-cvg-lem2-eq1}
    &-\langle \nabla f(\vx^k)-\nabla f(\vx^{k-t}), \nabla f(\vx^{k-t})\rangle \cr
\le &\textstyle \frac{\eta L_r}{2}\sum_{d=k-t}^{k-1}\Big(\sqrt{m}\|\nabla_{i_d}f(\hvx^{d})\|^2+\frac{1}{\sqrt{m}}\|\nabla f(\vx^{k-t})\|^2\Big).
\end{align}
By taking expectation, we have
\begin{align*}
\EE_{i_d, j_d} \|\nabla_{i_d}f(\hvx^{d})\|^2 
=  &\textstyle \frac{1}{m}\EE_{j_d} \|\nabla f(\vx^{d-j_d})\|^2\\
=  &\textstyle \frac{1}{m}\Big(\sum_{r=0}^{d-1} q_r\|\nabla f(\vx^{d-r})\|^2+c_d\|\nabla f(\vx^0)\|^2\Big).
\end{align*}
Now taking expectation on both sides of~\eqref{eq:lem:thm-cvg-lem2-eq1} and using the above equation, we get
\begin{align}\label{eq:cross1}
    &\EE [-\langle \nabla f(\vx^k)-\nabla f(\vx^{k-t}), \nabla f(\vx^{k-t})\rangle]\cr
\le &\textstyle  \frac{\eta L_r}{2\sqrt{m}} \sum_{d=k-t}^{k-1}\Big(\sum_{r=0}^{d-1} q_r\EE \|\nabla f(\vx^{d-r})\|^2+c_d\|\nabla f(\vx^0)\|^2\Big)\cr
    &+\textstyle  \frac{\eta L_r}{2\sqrt{m}} \sum_{d=k-t}^{k-1}t\EE\|\nabla f(\vx^{k-t})\|^2.
\end{align}
Finally,~\eqref{eq:thm-cvg-t1-eq} follows from
\begin{align*}
\textstyle \sum_{t=1}^{k-1} q_t\sum_{d=k-t}^{k-1}c_d\|\nabla f(\vx^0)\|^2\overset{(\ref{seq1})}
=  &\textstyle \sum_{d=1}^{k-1}\Big(\sum_{t=k-d}^{k-1} q_t\Big) c_d\|\nabla f(\vx^0)\|^2\\
=  &\textstyle \sum_{d=1}^{k-1} (c_{k-d}-c_k) c_d\|\nabla f(\vx^0)\|^2,
\end{align*}
and
\begin{align}\label{eq:gradr}
   &\textstyle \sum\limits_{t=1}^{k-1} q_t\sum\limits_{d=k-t}^{k-1}\sum\limits_{r=0}^{d-1} q_r\EE \|\nabla f(\vx^{d-r})\|^2\cr
=  &\textstyle \sum\limits_{d=1}^{k-1} (c_{k-d}-c_k)\sum\limits_{r=0}^{d-1} q_r\EE \|\nabla f(\vx^{d-r})\|^2\cr
[\text{let }r\leftarrow d-r]~=&\textstyle \sum\limits_{d=1}^{k-1} (c_{k-d}-c_k)\sum\limits_{r=1}^{d} q_{d-r}\EE \|\nabla f(\vx^{r})\|^2\cr
\overset{(\ref{seq2})}=&\textstyle \sum\limits_{r=1}^{k-1}\left(\sum\limits_{d=r}^{k-1}(c_{k-d}-c_k)q_{d-r}\right)\EE \|\nabla f(\vx^{r})\|^2\cr
[\text{let }t\leftarrow k-r,\ d\leftarrow k-d]~=&\textstyle \sum\limits_{t=1}^{k-1}\left(\sum\limits_{d=1}^t (c_{d}-c_k)q_{t-d}\right)\EE \|\nabla f(\vx^{k-t})\|^2.
\end{align}
\qed
\end{proof}
Using the above lemma, we show a result of running one iteration of the algorithm.
\begin{theorem}[Fundamental bound]\label{lem:thm-cvg-lem3}
Set $\gamma_k, \beta_k$ and $C_{t,k}$ as in~\eqref{def-para}. For any $k>1$, we have
\begin{align}
\textstyle \EE f(\vx^{k+1}) \le &  \EE f(\vx^k) + \gamma_k\|\nabla f(\vx^0)\|^2-\beta_k\EE\|\nabla f(\vx^k)\|^2\cr
 &\textstyle -\sum_{t=1}^{k-1} C_{t,k} \EE\|\nabla f(\vx^{k-t})\|^2.\label{eq:fx2}
\end{align}
\end{theorem}

\begin{proof}
Since $\vx^{k+1}=\vx^k-\eta U_{i_k}\nabla f(\vx^{k-j_k})$, we have from~\eqref{lip-ineq} that
$$f(\vx^{k+1})\le f(\vx^k)-\eta\langle \nabla f(\vx^k), U_{i_k}\nabla f(\vx^{k-j_k})\rangle+\tfrac{L_c}{2}\|\eta U_{i_k}\nabla f(\vx^{k-j_k})\|^2.$$
Taking conditional expectation on $(i_k, j_k)$ gives
\begin{align}\label{eq:thm-cvg-eq1}
    &\EE_{i_k,j_k}f(\vx^{k+1})\cr
\le &f(\vx^k)-\tfrac{\eta}{m}\EE_{j_k}\langle \nabla f(\vx^k), \nabla f(\vx^{k-j_k})\rangle+\tfrac{\eta^2 L_c}{2m}\EE_{j_k}\|\nabla f(\vx^{k-j_k})\|^2\cr
=   &\textstyle  f(\vx^k)-\tfrac{\eta}{m}\sum_{t=0}^{k-1} q_t\langle \nabla f(\vx^k), \nabla f(\vx^{k-t})\rangle-\tfrac{\eta}{m}c_k\langle \nabla f(\vx^k), \nabla f(\vx^0)\rangle\cr
    &\textstyle  +\tfrac{\eta^2 L_c}{2m}\sum_{t=0}^{k-1} q_t\|\nabla f(\vx^{k-t})\|^2+\tfrac{\eta^2 L_c}{2m}c_k\|\nabla f(\vx^0)\|^2.
\end{align}
For the first cross term in~\eqref{eq:thm-cvg-eq1}, we write each summand as
\begin{equation}\label{eq:thm-cvg-eq2}
\langle \nabla f(\vx^k), \nabla f(\vx^{k-t})\rangle=\langle \nabla f(\vx^k)-\nabla f(\vx^{k-t}), \nabla f(\vx^{k-t})\rangle+\|\nabla f(\vx^{k-t})\|^2,
\end{equation}
and we use Young's inequality to bound the second cross term by
\begin{equation}\label{eq:thm-cvg-eq3}
-\tfrac{\eta}{m}c_k\langle \nabla f(\vx^k), \nabla f(\vx^0)\rangle \le \tfrac{\eta c_k}{2m}\Big[\|\nabla f(\vx^k)\|^2+\|\nabla f(\vx^0)\|^2\Big].
\end{equation}
Now taking expectation over both sides of~\eqref{eq:thm-cvg-eq1}, plugging in~\eqref{eq:thm-cvg-eq2} and~\eqref{eq:thm-cvg-eq3}, and using Lemma~\ref{lem:thm-cvg-lem2}, we have the desired result. 
\qed\end{proof}
We are now ready to show the main result in the following theorem.
\begin{theorem}\label{thm:cvg}\textbf{\emph{(Convergence for the nonconvex smooth case)}}
Under Assumptions \ref{assump-sol} through \ref{assump-read}, let $\{\vx^k\}_{k\ge 1}$ be generated from Algorithm \ref{alg:asyn_bcd}. Assume $T<\infty$.
Take the stepsize as
$0<\eta<\frac{1/L_c}{1+2\kappa T/\sqrt{m}}.$ If $q_0>0$ or $\nabla f(\vx)$ is bounded for all $\vx$,
then
\begin{equation}\label{eq:grad-0}
\lim_{k\to\infty} \EE\|\nabla f(\vx^k)\|= 0,
\end{equation}
and any limit point of $\{\vx^k\}_{k\ge 1}$ is almost surely a critical point of~\eqref{eq:main-sm}.
\end{theorem}
\begin{remark}
If $T=\EE[\vj]=o(\sqrt{m})$, then $\eta$ only weakly depends on the delay. The conditions $q_0>0$ or $\nabla f(\vx)$ being bounded can be dropped if $S=\EE[\vj^2]$ is bounded; see Theorem \ref{thm:convg-nsm}.
\end{remark}

\begin{proof}
Summing up~\eqref{eq:fx2} from $k=0$ through $K$ and using~\eqref{seq3}, we have
{\small\begin{align}\label{eq:fxKJ}
 \textstyle \EE f(\vx^{K+1})\le &\textstyle f(\vx^0)+\sum_{k=0}^K\gamma_k\|\nabla f(\vx^0)\|^2\cr
&\textstyle -\beta_K\EE\|\nabla f(\vx^K)\|^2-\sum_{k=1}^{K-1}\left(\beta_k+\sum_{t=k+1}^K C_{t-k,t}\right) \EE\|\nabla f(\vx^k)\|^2.
\end{align}
}Note that $\beta_K\to\big(\frac{\eta}{m}-\frac{\eta^2 L_c}{2m}\big)q_0$ as $K\to\infty$. 
If $q_0>0$ or $\|\nabla f(\vx)\|$ is bounded, by letting $K\to\infty$ in~\eqref{eq:fxKJ} and using the lower boundedness of $f$, we have from Lemma \ref{lem:thm-cvg-lem1} that
$$
\textstyle \sum_{k=1}^\infty \left(\frac{\eta}{2m}-\frac{\eta^2 L_c}{2m}-\frac{\eta^2 L_r T}{m\sqrt{m}}\right)\EE\|\nabla f(\vx^k)\|^2<\infty.$$
Since $\eta<\frac{1/L_c}{1+2\kappa T/\sqrt{m}}$, we have~\eqref{eq:grad-0} from the above inequality.

From the Markov inequality, it follows that $\|\nabla f(\vx^k)\|$ converges to \emph{zero} with probability one. Let $\bar{\vx}$ be a limit point of $\{\vx^k\}_{k\ge1}$, i.e., there is a subsequence $\{\vx^k\}_{k\in\cK}$ convergent to $\bar{\vx}$. Hence, $\|\nabla f(\vx^k)\|\to 0$ almost surely as $\cK\ni k\to\infty$. By~\cite[Theorem 3.4, p.212]{gut2006probability}, there is a subsubsequence $\{\vx^k\}_{k\in\cK'}$ such that $\|\nabla f(\vx^k)\|\to 0$ almost surely as $\cK'\ni k\to\infty$. This completes the proof.
\qed\end{proof}

\subsection{Convergence rate for the convex case} In this subsection, we assume the convexity of $f$ and establish convergence rate results of Algorithm \ref{alg:asyn_bcd} for solving~\eqref{eq:main-sm}. Besides Assumptions \ref{assump-sol} through \ref{assump-read}, we make an additional assumption to the delay as follows. It means the delay follows a sub-exponential distribution.

\begin{assumption}\label{assump2}
There is a constant $\sigma>1$ such that
\begin{equation}\label{eq:cond-sig}
M_\sigma := \EE[\sigma^{\vj}]<\infty.
\end{equation}
\end{assumption}
The condition in~\eqref{eq:cond-sig} is stronger than $T<\infty$, and both of them hold if the delay $j_k$ is uniformly bounded by some number $\tau$ or follows the Poisson distribution; see the discussions in Section~\ref{sec:poisson}. Using this additional assumption and choosing an appropriate stepsize, we are able to control the gradient of $f$ such that it changes not too fast.

\begin{lemma}\label{lem:bdgrad}
Under Assumptions \ref{assump-fun} through \ref{assump2},  for any $1<\rho\le\sigma$, if the stepsize satisfies
\begin{equation}\label{sm-eta-rate}
0<\eta\le\tfrac{(\rho-1)\sqrt{m}}{\rho L_r(1+M_\rho)},
\end{equation}
with $M_\rho$ defined in~\eqref{eq:cond-sig}, then for all $k$, it holds that
\begin{equation}\label{eq:bdgrad}
\EE\|\nabla f(\vx^k)\|^2\le \rho \EE\|\nabla f(\vx^{k+1})\|^2 
\quad\text{and}\quad
\EE\|\nabla f(\vx^{k+1})\|^2\le \rho\EE\|\nabla f(\vx^k)\|^2.
\end{equation}
\end{lemma}

The proof of Lemma \ref{lem:bdgrad} follows an argument similar to~\cite{liu2014asynchronous}. Since it is rather long, it is included in the appendix. Similar to Lemma \ref{lem:thm-cvg-lem2}, we can show the following result.

\begin{lemma}\label{lem:thm-rate-cvx-sm-lem1}
For any $k$, it holds that
\begin{align}\label{eq:w-rate-fxk1}
&\textstyle \sum_{t=0}^{k-1} q_t\EE[-\langle \nabla f(\vx^k), \nabla f(\vx^{k-t})-\nabla f(\vx^k)\rangle]\cr
&-c_k\EE\langle \nabla f(\vx^k),\nabla f(\vx^0)-\nabla f(\vx^k)\rangle\cr
\le &\textstyle  \frac{\eta L_r}{2\sqrt{m}}\sum\limits_{d=1}^{k}c_{k-d}c_d\|\nabla f(\vx^0)\|^2+\frac{\eta L_r}{2\sqrt{m}}\sum\limits_{t=1}^{k-1}\sum\limits_{d=1}^t c_dq_{t-d}\EE\|\nabla f(\vx^{k-t})\|^2\nonumber\\
&\quad\textstyle +\frac{\eta L_r}{2\sqrt{m}}\left(\sum_{t=0}^{k-1} tq_t + kc_k\right)\EE\|\nabla f(\vx^k)\|^2.
\end{align}
\end{lemma}
\begin{proof}
Following an argument similar to how~\eqref{eq:cross1} is obtained, we can show
\begin{align*}
    &\textstyle \sum\limits_{t=0}^{k-1} q_t\EE[-\langle \nabla f(\vx^k), \nabla f(\vx^{k-t})-\nabla f(\vx^k)\rangle]\cr
\le &\textstyle \frac{\eta L_r}{2\sqrt{m}}\sum\limits_{t=0}^{k-1} q_t\Big(\sum\limits_{d=k-t}^{k-1}(\sum\limits_{r=0}^{d-1} q_r\EE \|\nabla f(\vx^{d-r})\|^2+c_d\|\nabla f(\vx^0)\|^2)+t\EE\|\nabla f(\vx^k)\|^2\Big),\nonumber\\
    &\textstyle -c_k\EE\langle \nabla f(\vx^k),\nabla f(\vx^0)-\nabla f(\vx^k)\rangle\cr
\le &\textstyle \frac{\eta L_r}{2\sqrt{m}}c_k\Big(\sum\limits_{d=0}^{k-1}\Big(\sum\limits_{r=0}^{d-1} q_r\EE \|\nabla f(\vx^{d-r})\|^2+c_d\|\nabla f(\vx^0)\|^2\Big)+k\EE\|\nabla f(\vx^k)\|^2\Big).
\end{align*}
Using the above inequalities, we complete the proof 
by noting~\eqref{eq:gradr},
\begin{align}
\textstyle \sum_{t=0}^{k-1} q_t\sum_{d=k-t}^{k-1}c_d+c_k\sum_{d=0}^{k-1}c_d
=  &\textstyle\sum_{d=1}^{k-1} (c_{k-d}-c_k) c_d+c_k\sum_{d=0}^{k-1}c_d\cr
=  &\textstyle\sum_{d=1}^{k-1}c_{k-d}c_d+c_k = \sum_{d=1}^{k}c_{k-d}c_d,
\end{align}
and
$
\textstyle c_k\sum_{d=0}^{k-1}\sum_{r=0}^{d-1} q_r \|\nabla f(\vx^{d-r})\|^2=\sum_{t=1}^{k-1}\sum_{d=1}^t c_k q_{t-d}\|\nabla f(\vx^{k-t})\|^2.
$
\qed\end{proof}
Using the above two lemmas, we establish sufficient objective decrease. 

\begin{theorem}[Sufficient progress]\label{thm:dec1}
Under Assumptions \ref{assump-sol} through \ref{assump2}, we let $\{\vx^k\}_{k\ge 1}$ be the sequence generated from Algorithm \ref{alg:asyn_bcd}. For a certain $1<\rho < \sigma$, define
\begin{equation}\label{N-rho}
N_\rho:=\EE[\vj\rho^{\vj}].
\end{equation}
Take the stepsize such that~\eqref{sm-eta-rate} is satisfied and also
\begin{equation}\label{step-eta2}
0<\eta< 
2\Big(L_c(M_\rho+\tfrac{\kappa(2N_\rho M_\rho+T)}{\sqrt{m}})\Big)^{-1}.
\end{equation}
Let
\begin{equation}\label{def-const-D}\textstyle D=\frac{\eta}{2m}\left(2-\tfrac{\eta L_r}{\sqrt{m}}(2N_\rho M_\rho+T)-\eta L_c M_\rho\right).
\end{equation}
Then,
\begin{equation}\label{dec}
\EE f(\vx^{k+1})\le \EE f(\vx^k) - D\EE\|\nabla f(\vx^k)\|^2.
\end{equation}
\end{theorem}

\begin{proof}
First note that for any $\rho<\sigma$, $t\rho^t$ is dominated by $\sigma^t$ as $t$ is sufficiently large. Hence, $N_\rho<\infty$ from~\eqref{eq:cond-sig}, and it is easy to see $T<\infty$. Also note that
\begin{equation}\label{N-rho2}
\textstyle\EE[\vj \rho^{\vj}]=\sum\limits_{t=1}^\infty t q_t \rho^t=\sum\limits_{t=1}^\infty \sum\limits_{d=1}^t q_t \rho^t=\sum\limits_{d=1}^\infty \sum\limits_{t=d}^\infty q_t \rho^t\geq\sum\limits_{d=1}^\infty \sum\limits_{t=d}^\infty q_t \rho^d=\sum\limits_{d=1}^\infty c_d \rho^d.
\end{equation}
We write the cross terms in~\eqref{eq:thm-cvg-eq1} to
$$\langle \nabla f(\vx^k), \nabla f(\vx^{k-t})\rangle=\langle \nabla f(\vx^k), \nabla f(\vx^{k-t})-\nabla f(\vx^k)\rangle + \|\nabla f(\vx^k)\|^2.$$
Taking expectation on both sides of~\eqref{eq:thm-cvg-eq1} and using~\eqref{eq:w-rate-fxk1}, we have
\begin{align}
\EE f(\vx^{k+1})
\le &\textstyle\EE f(\vx^k)+\frac{\eta^2 L_r}{2m\sqrt{m}}\sum\limits_{d=1}^{k}c_{k-d}c_d\|\nabla f(\vx^0)\|^2\nonumber\\
&\textstyle+\frac{\eta^2 L_r}{2m\sqrt{m}}\sum\limits_{t=1}^{k-1}\sum\limits_{d=1}^t c_dq_{t-d}\EE\|\nabla f(\vx^{k-t})\|^2\cr
&\textstyle+\frac{\eta^2 L_r}{2m\sqrt{m}}\left(\sum\limits_{t=0}^{k-1} tq_t + kc_k\right)\EE\|\nabla f(\vx^k)\|^2 -\frac{\eta}{m}\EE\|\nabla f(\vx^k)\|^2\nonumber\\
&\textstyle+\frac{\eta^2 L_c}{2m}\sum\limits_{t=0}^{k-1} q_t\EE\|\nabla f(\vx^{k-t})\|^2+\frac{\eta^2 L_c}{2m} c_k\|\nabla f(\vx^0)\|^2.\label{eq:w-rate-fxk0}
\end{align}
The above inequality together with~\eqref{eq:bdgrad} implies
\begin{align}\label{eq:w-rate-fxk2}
\EE f(\vx^{k+1})
\le &\textstyle\EE f(\vx^k)+\frac{\eta^2 L_r}{2m\sqrt{m}}\sum_{d=1}^{k}c_{k-d}c_d\rho^k\EE\|\nabla f(\vx^k)\|^2\nonumber \\
&\textstyle+\frac{\eta^2 L_r}{2m\sqrt{m}}\sum_{t=1}^{k-1}\sum_{d=1}^t c_dq_{t-d}\rho^t\EE\|\nabla f(\vx^{k})\|^2\cr
&\textstyle+\frac{\eta^2 L_r}{2m\sqrt{m}}\left(\sum_{t=0}^{k-1} tq_t + kc_k\right)\EE\|\nabla f(\vx^k)\|^2 -\frac{\eta}{m}\EE\|\nabla f(\vx^k)\|^2\nonumber\\
&\textstyle+\frac{\eta^2 L_c}{2m}\sum_{t=0}^{k-1} q_t\rho^t\EE\|\nabla f(\vx^{k})\|^2+\frac{\eta^2 L_c}{2m} c_k\rho^k\EE\|\nabla f(\vx^k)\|^2.
\end{align}
Note that $\sum_{t=1}^{k-1}\sum_{d=1}^t c_dq_{t-d}\rho^t\le \sum_{t=1}^\infty\sum_{d=1}^t c_dq_{t-d}\rho^t$, which by exchanging summations equals $\sum_{d=1}^\infty c_d\rho^d\sum_{t=d}^\infty q_{t-d}\rho^{t-d}\overset{(\ref{N-rho2})}\leq N_\rho M_\rho$. Also note that $\sum_{d=1}^{k}c_{k-d}c_d\rho^k=\sum_{d=1}^{k}c_{d}\rho^d c_{k-d}\rho^{k-d}\leq \sum_{d=1}^{k}c_{d}\rho^d\left(\sum_{r=0}^\infty q_r \rho^r\right)\le N_\rho M_\rho$.
From these relations and~\eqref{eq:w-rate-fxk2}, we obtain 
\begin{align*}
\EE f(\vx^{k+1})&\le\textstyle\EE f(\vx^k)+\frac{\eta^2 L_r}{m\sqrt{m}}N_\rho M_\rho\|\nabla f(\vx^k)\|^2 \\
&\textstyle+\frac{\eta^2 L_r}{2m\sqrt{m}}\left(\sum_{t=0}^{k-1} tq_t + kc_k\right)\EE\|\nabla f(\vx^k)\|^2 -\frac{\eta}{m}\|\nabla f(\vx^k)\|^2\cr
&\textstyle +\frac{\eta^2 L_c}{2m}\sum_{t=0}^{k-1} q_t\rho^t\EE\|\nabla f(\vx^{k})\|^2+\frac{\eta^2 L_c}{2m} c_k\rho^k\EE\|\nabla f(\vx^k)\|^2\cr
 & \textstyle\le\EE f(\vx^k)+\left(\frac{\eta^2 L_r}{2m\sqrt{m}}(2N_\rho M_\rho+T)+\frac{\eta^2 L_c}{2m}M_\rho-{\eta\over m}\right)\EE\|\nabla f(\vx^k)\|^2,
\end{align*}
which completes the proof.
\qed\end{proof}

Using~\eqref{dec} and the convexity of $f$, we  establish the following convergence rate. 

\begin{theorem}\textbf{\emph{(Convergence rate for the convex smooth case)}}\label{thm:rate-cvx-sm}
Under the assumptions of Theorem \ref{thm:dec1}, we have
\begin{enumerate}
\item If $f$ is convex and $\|\vx^k-\cP_{X^*}(\vx^k)\|\le B,\,\forall k$ for a certain constant $B$, then
\begin{align}\label{rate-wcvx-sm}
\textstyle\EE [f(\vx^{k+1})-f^*]
\le\frac{1}{(f(\vx^0)-f^*)^{-1}+(k+1)DB^{-2}},
\end{align}
where $f^*$ denotes the minimum value of~\eqref{eq:main-sm} and $D$ is given in~\eqref{def-const-D}.
\item If $f$ is strongly convex with constant $\mu$, then
\begin{equation}\label{sm-linear-rate}
\EE[f(\vx^{k+1})-f^*]\le (1-2\mu D)\EE[f(\vx^k)-f^*],
\end{equation}
where $D$ is given in~\eqref{def-const-D}.
\end{enumerate}
\end{theorem}
\begin{remark}
For the sublinear rate in~\eqref{rate-wcvx-sm}, 
we assume the boundedness of the iterates. This assumption can be relaxed if we use potentially smaller stepsize; see Theorem \ref{thm:ns-rate}.

For the linear convergence, the assumption on strongly convexity can be weakened to \emph{either essential or restrict strong convexity}, 
see~\cite{LaiYin2013_augmented} and~\cite{liu2014asynchronous}.
\end{remark}

\begin{proof}
If $\|\vx^k-\cP_{X^*}(\vx^k)\|\le B$, then from $f(\vx^k)-f(\cP_{X^*}(\vx^k))\le \langle\nabla f(\vx^k),\vx^k-\cP_{X^*}(\vx^k)\rangle$, we have
$$|f(\vx^k)-f^*|\le \|\nabla f(\vx^k)\|\cdot\|\vx^k-\cP_{X^*}(\vx^k)\|\le B\|\nabla f(\vx^k)\|,$$
and thus
\begin{equation}\label{bd-iter}
\textstyle\|\nabla f(\vx^k)\|^2\ge \frac{1}{B^2}(f(\vx^k)-f^*)^2.
\end{equation}
Substituting~\eqref{bd-iter} into~\eqref{dec} yields
$$\textstyle\EE f(\vx^{k+1})\le \EE f(\vx^k)-\frac{D}{B^2}\EE(f(\vx^k)-f^*)^2.$$
Hence,
\begin{align*}
&\textstyle~\EE [f(\vx^{k+1})-f^*] \le \EE [f(\vx^k)-f^*]-\frac{D}{B^2}\EE (f(\vx^k)-f^*)^2\\
\Rightarrow\, &\textstyle~\frac{1}{\EE [f(\vx^{k+1})-f^*]}\ge \frac{1}{\EE [f(\vx^k)-f^*]}+\frac{D}{B^2}\frac{\EE [f(\vx^k)-f^*]}{\EE [f(\vx^{k+1})-f^*]}\ge\frac{1}{\EE [f(\vx^k)-f^*]}+\frac{D}{B^2}\\
\Rightarrow\, &\textstyle~\frac{1}{\EE [f(\vx^{k+1})-f^*]}\ge\frac{1}{[f(\vx^0)-f^*]}+\frac{D(k+1)}{B^2},
\end{align*}
and thus~\eqref{rate-wcvx-sm} holds.

If $f$ is strongly convex with constant $\mu$, then
$$-\tfrac{1}{2\mu}\|\nabla f(\vx^k)\|^2\le f^*-f(\vx^k).$$
We immediately have~\eqref{sm-linear-rate} from~\eqref{dec} and the above inequality. This completes the proof.
\hfill
\qed\end{proof}

\section{Convergence results for the nonsmooth case} In this section, we analyze the convergence of Algorithm \ref{alg:asyn_bcd} for possibly nonsmooth cases. Throughout this section, we let
$$\bar{\vx}^{k+1}=\prox_{\eta R}\left(\vx^k-\eta\nabla f(\vx^{k-j_k})\right)$$
a virtual full-update iterate, where $R$ is defined in~\eqref{def-R}, and denote 
$$\vd^{k} = \bar\vx^{k+1}-\vx^k.$$
Due to more generality, we will make stronger assumptions on the delay than those made in the previous section. But all these assumptions are satisfied if the delay is uniformly bounded or follows the Poisson distribution, as shown in Section~\ref{sec:poisson}.

\subsection{Convergence for the nonconvex case} We first establish the almost sure global convergence for possibly nonconvex cases starting with the following square summable result.

\begin{lemma}[Square summability]\label{lem:ns-sqsum}
Under Assumptions~\ref{assump-sol} through~\ref{assump-read}, we let $\{\vx^k\}_{k\ge 1}$ be the sequence generated in Algorithm \ref{alg:asyn_bcd}.
Assume $S<\infty$,
and the stepsize is taken as
$0<\eta< \frac{1/L_c}{1+\kappa^2S/(2m)}.$
Then \begin{equation}\label{ns-sqsum}
\sum_{k=0}^\infty\EE\|\vd^k\|^2<\infty.
\end{equation}
\end{lemma}
\begin{proof}
By the definition of $\bar{\vx}^{k+1}$, we have
$-\nabla f(\vx^{k-j_k})-\tfrac{1}{\eta}\vd^k\in\partial R(\bar{\vx}^{k+1})$,
which together with the convexity of $R$ implies that, for any $\vx$,
\begin{equation}\label{eq:ineq-R}
R(\bar{\vx}^{k+1})-R(\vx)\le -\langle \nabla f(\vx^{k-j_k})+\tfrac{1}{\eta}\vd^k,\bar{\vx}^{k+1}-\vx\rangle.
\end{equation}
By $\vx^{k+1}=\vx^k+U_{i_k}\vd^k$ and~\eqref{lip-ineq}, we get
$F(\vx^{k+1})\le f(\vx^k)+\langle\nabla_{i_k}f(\vx^k),\vd^k_{i_k}\rangle+\tfrac{L_c}{2}\|{\vd}^{k}_{i_k}\|^2+R(\vx^{k+1}).$
To this inequality, take conditional expectation on $i_k$:
$$\textstyle\EE_{i_k} F(\vx^{k+1})\le F(\vx^k)+\frac{1}{m}\left(\langle\nabla f(\vx^k), \vd^k\rangle+\frac{L_c}{2}\|\vd^k\|^2+R(\bar{\vx}^{k+1})-R(\vx^k)\right).$$
To bound the right-hand side, we split the cross term as $$\langle\nabla f(\vx^k), \vd^k\rangle=\langle\nabla f(\vx^{k-j_k}), \vd^k\rangle+\langle\nabla f(\vx^k)-\nabla f(\vx^{k-j_k}),\vd^k\rangle$$ and apply~\eqref{eq:ineq-R} with $\vx=\vx^k$, arriving at
\begin{align}\label{ns-ineq1}
\textstyle\EE_{i_k} F(\vx^{k+1})
\le &\textstyle F(\vx^k)+\frac{1}{m}\big(\frac{L_c}{2}-\frac{1}{\eta}\big)\|\vd^k\|^2+\frac{1}{m}\langle\nabla f(\vx^k)-\nabla f(\vx^{k-j_k}),\vd^k\rangle.
\end{align}
Following a similar argument in the proof of Lemma \ref{lem:thm-cvg-lem2} and Young's inequality, we get
\begin{align}\nonumber
\textstyle \langle\nabla f(\vx^k)-\nabla f(\vx^{k-j_k}),\vd^k\rangle\le &\textstyle L_r\sum\limits_{d=k-j_k}^{k-1}\|\vx^{d+1}-\vx^d\|\cdot\|\vd^k\|\\
  \le &\textstyle\frac{L_r}{2\kappa}\|\vd^k\|^2+\frac{\kappa L_r}{2} \Big(j_k\sum\limits_{d=k-j_k}^{k-1}\|\vx^{d+1}-\vx^d\|^2\Big).\label{eq:pf-lem:ns-sqsum-eq2}
\end{align}
Note that
\begin{align}\label{diff-zero}
 &\textstyle \EE\Big[j_k\sum\limits_{d=k-j_k}^{k-1}\|\vx^{d+1}-\vx^d\|^2\Big]\cr
=&\textstyle  \sum\limits_{t=1}^{k-1} q_t t \sum\limits_{d=k-t}^{k-1}\EE\|\vx^{d+1}-\vx^d\|^2+ \sum\limits_{t=k}^\infty q_tt\sum\limits_{d=0}^{k-1}\EE\|\vx^{d+1}-\vx^d\|^2\cr
=&\textstyle \frac{1}{m}\sum\limits_{t=1}^{k-1} q_t t \sum\limits_{d=k-t}^{k-1}\EE\|{\vd}^{d}\|^2+\frac{1}{m} \sum\limits_{t=k}^\infty q_tt\sum\limits_{d=0}^{k-1}\EE\|\vd^d\|^2.
\end{align}
Hence, 
taking expectation yields
\begin{align}\label{ns-ineq2}
&\EE\langle\nabla f(\vx^k)-\nabla f(\vx^{k-j_k}),\vd^k\rangle\cr
&\hspace{-10pt}\le\textstyle \frac{L_r}{2}\Big[\frac{1}{\kappa}\EE\|\vd^k\|^2+\frac{\kappa}{m}\Big(\sum\limits_{t=1}^{k-1} q_t t \sum\limits_{d=k-t}^{k-1}\EE\|\vd^d\|^2+\sum\limits_{t=k}^{\infty} q_t t \sum\limits_{d=0}^{k-1}\EE\|\vd^d\|^2\Big)\Big].
\end{align}
Taking expectation on both sides of~\eqref{ns-ineq1} and substituting~\eqref{ns-ineq2} yield
\begin{align}\label{ns-ineq3}
&\textstyle \EE[F(\vx^{k+1})-F(\vx^k)]+\frac{1}{m}\left(\frac{1}{\eta}-L_c\right)\EE\|\vd^k\|^2\cr
\le &\textstyle \frac{\kappa L_r}{2m^2}\Big(\sum\limits_{t=1}^{k-1} q_t t \sum\limits_{d=k-t}^{k-1}\EE\|\vd^d\|^2+\sum\limits_{t=k}^{\infty} q_t t \sum\limits_{d=0}^{k-1}\EE\|\vd^d\|^2\Big).
\end{align}
From Lemma \ref{lem:2seq}, we have that for any $K\ge0$,
\begin{align}\label{ns-ineq4}
\textstyle \sum\limits_{k=0}^K\sum\limits_{t=1}^{k-1} q_t t \sum\limits_{d=k-t}^{k-1}\EE\|\vd^d\|^2
\overset{(\ref{seq1})}=&\textstyle \sum\limits_{k=0}^K\sum\limits_{d=1}^{k-1}\Big(\sum\limits_{t=k-d}^{k-1}q_t t\Big)\EE\|\vd^d\|^2\cr
\overset{(\ref{seq2})}=&\textstyle \sum\limits_{d=1}^{K-1}\sum\limits_{k=d+1}^K\Big(\sum\limits_{t=k-d}^{k-1}q_t t\Big)\EE\|\vd^d\|^2\cr
[k\leftrightarrow d]\quad=&\textstyle \sum\limits_{k=1}^{K-1}\Big(\sum\limits_{d=k+1}^K\sum\limits_{t=d-k}^{d-1}q_t t\Big)\EE\|\vd^k\|^2,\\
\text{and}\qquad\textstyle\sum\limits_{k=0}^K\sum\limits_{t=k}^{\infty} q_t t \sum\limits_{d=0}^{k-1}\EE\|\vd^d\|^2
=&\textstyle\sum\limits_{k=1}^K\sum\limits_{d=0}^{k-1}\Big(\sum\limits_{t=k}^{\infty} q_t t\Big)\EE\|\vd^k\|^2\cr
\overset{(\ref{seq2})}=&\textstyle\sum\limits_{d=0}^{K-1}\sum\limits_{k=d+1}^K\Big(\sum\limits_{t=k}^{\infty} q_t t\Big)\EE\|\vd^d\|^2\cr
[k\leftrightarrow d]\quad=&\textstyle\sum\limits_{k=0}^{K-1}\Big(\sum\limits_{d=k+1}^K\sum\limits_{t=d}^{\infty} q_t t\Big)\EE\|\vd^k\|^2.\label{ns-ineq5}
\end{align}
Summing up~\eqref{ns-ineq3} from $k=0$ through $K$ and substituting~\eqref{ns-ineq4} and~\eqref{ns-ineq5}, we have
\begin{align}\label{ns-ineq6}
   &\textstyle \EE[F(\vx^{K+1})-F(\vx^0)]+\frac{1}{m}\big(\frac{1}{\eta}-{L_c}\big)\sum\limits_{k=0}^K\EE\|\vd^k\|^2 \cr
\le&\textstyle\frac{\kappa L_r}{2m^2}\sum\limits_{k=0}^{K-1}\Big(\sum\limits_{d=k+1}^K\sum\limits_{t=d-k}^{\infty} q_t t\Big)\EE\|\vd^k\|^2.
\end{align}
Note that
$$\textstyle\sum\limits_{d=k+1}^K\sum\limits_{t=d-k}^{\infty} q_t t=\sum\limits_{d=1}^{K-k}\sum\limits_{t=d}^{\infty} q_t t\le\sum\limits_{d=1}^\infty\sum\limits_{t=d}^\infty q_t t=\sum\limits_{t=1}^\infty t^2q_t=S.$$
Since $F$ is lower bounded, we have~\eqref{ns-sqsum} from~\eqref{ns-ineq6} by letting $K\to\infty$.
\qed\end{proof}
Since $\big(\EE[\vj]\big)^2\le \EE[\vj^2]$, the condition $S<\infty$ implies $T<\infty$. Equation~\eqref{ns-sqsum} indicates that $\EE\|\vd^k\|\to0$ as $k\to\infty$. Together with $S<\infty$, we are able to show $\EE\|\vx^k-\vx^{k-j_k}\|$ also approaches \emph{zero}, as summarized in the following.

\begin{lemma}\label{lem:kjk}
Under the assumptions of Lemma \ref{lem:ns-sqsum}, we have
$$\lim_{k\to\infty}\EE\|\vx^k-\vx^{k-j_k}\|=0.$$
\end{lemma}
\begin{proof}
Pick any $\epsilon>0$. From~\eqref{ns-sqsum}, there must exist an integer $J>0$ such that
\begin{equation}\label{kjk-ineq1}
\textstyle\sum\limits_{d=J}^\infty\EE\|\vd^d\|^2\le m\epsilon\Big(3\sum\limits_{t=1}^\infty q_t t\Big)^{-1}.
\end{equation}
For the above $J$, there must exist an integer $K>J$ such that, for any $k\ge K$,
\begin{equation}\label{kjk-ineq2}
\textstyle\sum\limits_{t=k-J}^\infty q_t t\le{m\epsilon}{\Big(3\sum\limits_{d=0}^\infty\EE\|\vd^d\|^2\Big)^{-1}}.
\end{equation}
From Young's inequality, it follows that $\|\vx^k-\vx^{k-j_k}\|^2\le j_k\sum_{d=k-j_k}^{k-1}\|\vx^{d+1}-\vx^d\|^2.$
Hence, for any $k\ge K$, using~\eqref{diff-zero} and~\eqref{seq1}, we have
\begin{align*}
\EE\|\vx^k-\vx^{k-j_k}\|^2
\le&\textstyle\frac{1}{m}\Big[\sum\limits_{d=1}^{k-1}\Big(\sum\limits_{t=k-d}^{k-1}q_t t\Big)\EE\|\vd^d\|^2+\sum\limits_{d=0}^{k-1}\Big(\sum\limits_{t=k}^\infty q_t t\Big)\EE\|\vd^d\|^2\Big]\cr
=&\textstyle\frac{1}{m}\sum\limits_{d=1}^{J}\Big(\sum\limits_{t=k-d}^{k-1}q_t t\Big)\EE\|\vd^d\|^2\cr
&\textstyle+\frac{1}{m}\Big[\sum\limits_{d=J+1}^{k-1}\Big(\sum\limits_{t=k-d}^{k-1}q_t t\Big)\EE\|\vd^d\|^2+\sum\limits_{d=0}^{k-1}\Big(\sum\limits_{t=k}^\infty q_t t\Big)\EE\|\vd^d\|^2\Big]\cr
\le &\textstyle\frac{1}{m}\sum\limits_{d=1}^{J}\Big(\sum\limits_{t=k-J}^\infty q_t t\Big)\EE\|\vd^d\|^2\cr
&\textstyle+\frac{1}{m}\Big[\sum\limits_{d=J+1}^{k-1}\Big(\sum\limits_{t=1}^\infty q_t t\Big)\EE\|\vd^d\|^2+\sum\limits_{d=0}^{k-1}\Big(\sum\limits_{t=k-J}^\infty q_t t\Big)\EE\|\vd^d\|^2\Big],
\end{align*}
which implies $\EE\|\vx^k-\vx^{k-j_k}\|^2\le \epsilon$ under~\eqref{kjk-ineq1} and~\eqref{kjk-ineq2}. We have
$\lim_{k\to\infty}\EE\|\vx^k-\vx^{k-j_k}\|^2=0$ as  $\epsilon$ is arbitrary.  Now note
$\EE\|\vx^k-\vx^{k-j_k}\|\le \sqrt{\EE\|\vx^k-\vx^{k-j_k}\|^2}$
to complete the proof.
\qed\end{proof}
Using Lemmas \ref{lem:ns-sqsum} and \ref{lem:kjk}, we  establish the almost sure global convergence of Algorithm \ref{alg:asyn_bcd}.
\begin{theorem}\label{thm:convg-nsm}
Under the assumptions of Lemma \ref{lem:ns-sqsum}, any limit point $\vx^*$ of $\{\vx^k\}$ is a critical point of~\eqref{eq:main} almost surely.
\end{theorem}
Before proving this theorem, we make two remarks as follows.
\begin{remark}
From the theorem, we see that if $S=\EE[\vj^2]=o(m)$, then the  stepsize required for convergence only weakly depends on the delay.
\end{remark}

\begin{remark}[Comparison of stepsize]
The works~\cite{davis2016asynchronous} consider asynchronous coordinate descent for nonconvex problems. To have convergence to critical points, they assume delays bounded by a number $\tau$. Also, they require the boundedness of iterates and the stepsize less than $\frac{1/L_c}{1+2\kappa\tau/\sqrt{m}}$. Note that our stepsize in Theorem~\ref{thm:convg-nsm} is larger if $\kappa^2 S\le 16m$, where $S=\EE[\vj^2]<\tau^2$, and that can lead to faster convergence.
\end{remark}

\begin{proof}
Let $\{\vx^k\}_{k\in\cK}$ be a subsequence that converges to $\vx^*$. Since $\EE\|\vd^k\|\to 0$ as $\cK\ni k\to\infty$, from the Markov inequality, $\|\vd^k\|$ converges to \emph{zero} in probability as $\cK\ni k\to\infty$. By~\cite[Theorem 3.4, pp.212]{gut2006probability}, there is a subsubsequence $\{\vx^k\}_{k\in\cK'}$ such that $\|\vd^k\|$ almost surely converges to \emph{zero} as $\cK'\ni k\to\infty$. Hence, $\bar{\vx}^{k+1}$ almost surely converges to $\vx^*$ as $\cK'\ni k\to\infty$.

Since $-\nabla f(\vx^{k-j_k})-\frac{1}{\eta}\vd^k\in\partial R(\bar{\vx}^{k+1})$, we have
$$\text{dist}\big(\vzero,\partial F(\bar{\vx}^{k+1})\big)\le \|\nabla f(\bar{\vx}^{k+1})-\nabla f(\vx^{k-j_k})-\tfrac{1}{\eta}\vd^k\|.$$
Using triangle inequality and the Lipschitz continuity of $\nabla f$, and taking expectation give
$$\EE\text{dist}\big(\vzero,\partial F(\bar{\vx}^{k+1})\big)\le L_f \EE\|\vd^k\|+L_f\EE\|\vx^k-\vx^{k-j_k}\|+\tfrac{1}{\eta}\EE\|\vd^k\|.$$
From Lemmas \ref{lem:ns-sqsum} and \ref{lem:kjk}, it follows that the right-hand side  approaches to \emph{zero} as $k\to \infty$. Hence, $\EE\text{dist}\big(\vzero,\partial F(\bar{\vx}^{k+1})\big)\to 0$ as $k\to\infty$.
If necessary, passing to another subsequence, we use Markov inequality and~\cite[Theorem 3.4, pp.212]{gut2006probability} again to have $\text{dist}\big(\vzero,\partial F(\bar{\vx}^{k+1})\big)$ almost surely converges to \emph{zero} as $\cK'\ni k\to\infty$. Now use the outer semicontinuity~\cite{rockafellar2009variational} of $\text{dist}\big(\vzero,\partial F(\vx)\big)$ to obtain the desired result.
\hfill\qed\end{proof}

\subsection{Convergence rate for the convex case} In this subsection, we establish convergence rates of Algorithm \ref{alg:asyn_bcd} for nonsmooth convex cases. Similar to~\eqref{eq:bdgrad}, we first show that choosing an appropriate stepsize, the iterate difference does not change too fast.

\begin{lemma}[Fundamental bounds]\label{lem:nsm-dec}
Assume Assumptions \ref{assump-fun} through \ref{assump2}. Then for any $1<\rho<\sigma$, it holds that
\begin{equation}\label{def-gamma}
\gamma_{\rho, 1}:=\sum_{t=1}^\infty q_t \tfrac{\rho^{t/2}-1}{\rho^{1/2}-1}<\infty 
\quad\text{and}\quad
\gamma_{\rho,2}:=\Big(\sum_{t=1}^\infty q_t t\tfrac{\rho^{t}-1}{1-\rho^{-1}}\Big)^{1/2}<\infty.
\end{equation}
In addition, if the stepsize is taken such that
\begin{equation}\label{ns-rate-eta}
\textstyle 0<\eta\le\frac{(1-\rho^{-1})\sqrt{m}-4}{2L_r(1+\gamma_{\rho,1}+\gamma_{\rho,2})},
\end{equation}
then, for all $k\ge1$,
\begin{equation}\label{dec-xbar}
\EE\|\vd^{k-1}\|^2\le \rho\EE\|\vd^k\|^2.
\end{equation}
\end{lemma}
\begin{proof} It is easy to show~\eqref{def-gamma} by noting that $t\rho^t$ is dominated by $\sigma^t$ as $t$ is sufficiently large. Next we show~\eqref{dec-xbar} by induction.

Using the inequality $\|\vu\|^2-\|\vv\|^2\le 2\|\vu\|\cdot\|\vv-\vu\|$, we have
\begin{equation}\label{ineq1-dx}
\|\vd^{k-1}\|^2-\|\vd^{k}\|^2\le 2\|\vd^{k-1}\|\cdot\|\vd^{k}-\vd^{k-1}\|,\,\forall k.
\end{equation}
In addition, for all $k$,
\begin{align}\label{ineq2-dx}
\EE\|\vx^{k-1}-\vx^k\|\|\vd^{k-1}\|
\le&\tfrac{1}{2}\EE\left[\sqrt{m}\|\vx^{k-1}-\vx^k\|^2+\tfrac{1}{\sqrt{m}}\|\vd^{k-1}\|^2\right] \cr
=&\tfrac{1}{\sqrt{m}}\EE\|\vd^{k-1}\|^2.
\end{align}
Furthermore, from $\vd^{k}-\vd^{k-1}=\vx^k-\prox_{\eta R}\big(\vx^k-\eta\nabla f(\vx^{k-j_k})\big)-\vx^{k-1}+\prox_{\eta R}\big(\vx^{k-1}-\eta\nabla f(\vx^{k-1-j_{k-1}})\big)$, the nonexpansiveness of $\prox_{\eta R}$, and the triangle inequality, we have
\begin{align}
   &\|\vd^{k}-\vd^{k-1}\|\, \cr
\le&\|\vx^k-\vx^{k-1}\|+\|\vx^k-\eta\nabla f(\vx^{k-j_k})-\vx^{k-1}+\eta\nabla f(\vx^{k-1-j_{k-1}})\|\nonumber\\
\le& 2\|\vx^k-\vx^{k-1}\|+\eta\|\nabla f(\vx^{k-j_k})-\nabla f(\vx^{k-1-j_{k-1}})\|\label{ineq3-dx}\\
\le& 2\|\vx^k-\vx^{k-1}\|+\eta\|\nabla f(\vx^{k-j_k})-\nabla f(\vx^{k})\| \cr
   & +\eta\|\nabla f(\vx^{k})-\nabla f(\vx^{k-1-j_{k-1}})\|.\label{ineq4-dx}
\end{align}
When $k=1$, we have $j_0=0$ and $j_1\in\{0, 1\}$ because $j_k\le k,\,\forall k$. Hence, from~\eqref{ineq3-dx}, 
$$\|\vd^1-\vd^0\|\le 2\|\vx^1-\vx^0\|+\eta\|\nabla f(\vx^1)-\nabla f(\vx^0)\|\le (2+\eta L_r)\|\vx^1-\vx^0\|,$$
which together with~\eqref{ineq1-dx} and~\eqref{ineq2-dx} implies
$$\textstyle\EE\big[\|\vd^0\|^2-\|\vd^1\|^2\big]\le (4+2\eta L_r)\EE\big[\|\vd^0\|\cdot\|\vx^0-\vx^1\|\big]\le\frac{4+2\eta L_r}{\sqrt{m}}\EE \|\vd^0\|^2.$$
Hence,
$$\textstyle\EE\|\vd^0\|^2\le\big(1-\frac{4+2\eta L_r}{\sqrt{m}}\big)^{-1}\EE\|\vd^1\|^2\overset{(\ref{ns-rate-eta})}\le \rho \EE\|\vd^1\|^2.$$
Assume~\eqref{dec-xbar} holds for all $k\le K-1$. We show it holds for $k=K$. First, for any $d\le K-1$,
\begin{align}\label{ineq5-dx}
\EE\|\vd^{K-1}\|\cdot\|\vx^d-\vx^{d+1}\|
\le&\textstyle\frac{1}{2}\EE\Big[\frac{\rho^{\frac{K-1-d}{2}}}{\sqrt{m}}\|\vd^{K-1}\|^2+\frac{\sqrt{m}}{\rho^{\frac{K-1-d}{2}}}\|\vx^d-\vx^{d+1}\|^2\big]\cr
=&\textstyle\frac{1}{2}\EE\Big[\frac{\rho^{\frac{K-1-d}{2}}}{\sqrt{m}}\|\vd^{K-1}\|^2+\frac{1}{\sqrt{m}\rho^{\frac{K-1-d}{2}}}\|\vd^d\|^2\Big]\cr
\le &\textstyle\frac{1}{2}\EE\Big[\frac{\rho^{\frac{K-1-d}{2}}}{\sqrt{m}}\|\vd^{K-1}\|^2+\frac{\rho^{K-1-d}}{\sqrt{m}\rho^{\frac{K-1-d}{2}}}\|\vd^{K-1}\|^2\Big]\cr
=&\textstyle\frac{\rho^{\frac{K-1-d}{2}}}{\sqrt{m}}\EE\|\vd^{K-1}\|^2.
\end{align}
Secondly, we have
\begin{align}\label{ineq6-dx}
&\textstyle\EE\big[\|\vd^{K-1}\|^2-\|\vd^{K}\|^2\big] 
\overset{(\ref{ineq1-dx})} \le 2\EE\|\vd^{K-1}\|\|\vd^{K}-\vd^{K-1}\|\cr
\overset{(\ref{ineq4-dx})}\le & 4\EE \|\vd^{K-1}\|\|\vx^K-\vx^{K-1}\| +2\eta\EE \|\vd^{K-1}\|\|\nabla f(\vx^{K})-\nabla f(\vx^{K-1})\|  \cr
\textstyle &+2\eta\EE \|\vd^{K-1}\|\|\nabla f(\vx^{K-j_K})-\nabla f(\vx^{K})\| \cr
\textstyle &+2\eta\EE \|\vd^{K-1}\|\|\nabla f(\vx^{K-1})-\nabla f(\vx^{K-1-j_{K-1}})\| \cr
\overset{(\ref{ineq2-dx})}\le &\tfrac{4+2\eta L_r}{\sqrt{m}}\EE\|\vd^{K-1}\|^2+2\eta\EE \|\vd^{K-1}\|\|\nabla f(\vx^{K-j_K})-\nabla f(\vx^{K})\|\cr
\textstyle&+2\eta\EE \|\vd^{K-1}\|\|\nabla f(\vx^{K-1})-\nabla f(\vx^{K-1-j_{K-1}})\|.
\end{align}
Note that
\begin{align*}
\textstyle \EE_{j_K}\|\nabla f(\vx^{K-j_K})-\nabla f(\vx^{K})\|=&\textstyle\sum\limits_{t=1}^{K-1}q_t\|\nabla f(\vx^{K-t})-\nabla f(\vx^{K})\|\\
&+c_K\|\nabla f(\vx^0)-\nabla f(\vx^{K})\|.
\end{align*}
By the triangle inequality and  the Lipschitz  of $\nabla f$, it follows that, for any $1\le t\le K$,
\begin{align}\label{eq:tri-lip-ineq}\textstyle\|\nabla f(\vx^{K-t})-\nabla f(\vx^{K})\|\le &\textstyle\sum_{d=K-t}^{K-1}\|\nabla f(\vx^d)-\nabla f(\vx^{d+1})\| \cr
\le &\textstyle L_r \sum_{d=K-t}^{K-1}\|\vx^d-\vx^{d+1}\|.
\end{align}
Since $\|\vd^{K-1}\|$ is independent of $j_K$, we have from the above two equations that
\begin{align*}
   & \EE \|\vd^{K-1}\|\cdot\|\nabla f(\vx^{K-j_K})-\nabla f(\vx^{K})\| \\
\le &\textstyle L_r\sum\limits_{t=1}^{K-1}q_t\EE\|\vd^{K-1}\|\sum\limits_{d=K-t}^{K-1}\|\vx^d-\vx^{d+1}\|
    +L_r c_K\EE\|\vd^{K-1}\|\sum\limits_{d=0}^{K-1}\|\vx^d-\vx^{d+1}\|.
\end{align*}
Using~\eqref{ineq5-dx},  the definition of $\gamma_{\rho,1}$ in~\eqref{def-gamma} and
$\textstyle\sum_{d=K-t}^{K-1} \rho^{\frac{K-1-d}{2}}=\frac{\rho^{t/2}-1}{\rho^{1/2}-1},\,\forall 1\le t\le K,$
we have
\begin{align}\label{ineq6-1-dx}
\EE \|\vd^{K-1}\|\|\nabla f(\vx^{K-j_K})-\nabla f(\vx^{K})\|
\le \tfrac{L_r}{\sqrt{m}}\gamma_{\rho,1}\EE \|\vd^{K-1}\|^2.
\end{align}
Also, using  Young's inequality and~\eqref{eq:tri-lip-ineq} with $K$ replaced by $K-1$ and $t=j_{K-1}$, we have, for any $\beta>0$,
\begin{align}\label{eq:tri-lip-ineq-2}
&\EE \|\vd^{K-1}\|\|\nabla f(\vx^{K-1})-\nabla f(\vx^{K-1-j_{K-1}})\|\cr
&\le  \textstyle \frac{L_r}{2\beta}\EE \|\vd^{K-1}\|^2+\frac{L_r\beta}{2}\EE\Big[\sum\limits_{d=K-1-j_{K-1}}^{K-2}\|\vx^{d} - \vx^{d+1}\|\Big]^2.
\end{align}
Note that
\begin{align*}
    &\textstyle\EE\Big[\sum\limits_{d=K-1-j_{K-1}}^{K-2}\|\vx^{d} - \vx^{d+1}\|\Big]^2\cr
=   &\textstyle\sum\limits_{t=1}^{K-2}q_t\EE\Big[\sum\limits_{d=K-1-t}^{K-2}\|\vx^{d} - \vx^{d+1}\|\Big]^2+c_{K-1}\EE\Big[\sum\limits_{d=0}^{K-2}\|\vx^{d} - \vx^{d+1}\|\Big]^2\cr
\le &\textstyle\sum\limits_{t=1}^{K-2}q_t t\sum\limits_{d=K-1-t}^{K-2}\EE\|\vx^{d} - \vx^{d+1}\|^2+c_{K-1}(K-1)\sum\limits_{d=0}^{K-2}\EE\|\vx^{d} - \vx^{d+1}\|^2.
\end{align*}
Substituting this inequality into~\eqref{eq:tri-lip-ineq-2}, noting $\EE\|\vx^d-\vx^{d+1}\|^2=\frac{1}{m}\EE\|\vd^d\|^2$, and  applying~\eqref{dec-xbar}  for all $k\le K-1$, we have
\begin{align*}
\EE \|\vd^{K-1}\|\|\nabla f(\vx^{K-1})-\nabla f(\vx^{K-1-j_{K-1}})\| \le C\EE \|\vd^{K-1}\|^2
\end{align*}
where $C= \frac{L_r}{2\beta}+\frac{L_r\beta}{2m}\sum\limits_{t=1}^{K-2}q_t t\sum\limits_{d=K-1-t}^{K-2}\rho^{K-1-d}+\frac{L_r\beta}{2m}c_{K-1}(K-1)\sum\limits_{d=0}^{K-2}\rho^{K-1-d}$.
Now let $\beta={\sqrt{m}}{\big(\sum_{t=1}^{K-2}q_tt\frac{\rho^{t}-1}{1-\rho^{-1}}+c_{K-1}(K-1)\frac{\rho^{K-1}-1}{1-\rho^{-1}}\big)^{-1/2}}$ and recall the definition of $\gamma_{\rho,2}$ in~\eqref{def-gamma}. From the above inequality, we have
\begin{align}\label{ineq6-2-dx}\textstyle
\EE \|\vd^{K-1}\|\|\nabla f(\vx^{K-1})-\nabla f(\vx^{K-1-j_{K-1}})\|
\le  \frac{L_r\gamma_{\rho,2}}{\sqrt{m}}\EE \|\vd^{K-1}\|^2.
\end{align}
Substituting~\eqref{ineq6-1-dx} and~\eqref{ineq6-2-dx} into~\eqref{ineq6-dx} gives
$$\textstyle\EE\big[\|\vd^{K-1}\|^2-\|\vd^{K}\|^2\big]\le\frac{4+2\eta L_r(1+\gamma_{\rho,1}+\gamma_{\rho,2})}{\sqrt{m}}\EE \|\vd^{K-1}\|^2,$$
and thus
$$\textstyle \EE \|\vd^{K-1}\|^2\le \left(1-\frac{4+2\eta L_r(1+\gamma_{\rho,1}+\gamma_{\rho,2})}{\sqrt{m}}\right)^{-1}\EE\|\vd^{K}\|^2\overset{(\ref{ns-rate-eta})}\le \rho \EE\|\vd^{K}\|^2.$$
Therefore, by induction, it follows that~\eqref{dec-xbar} holds for all $k$, and we complete the proof.\qed\end{proof}
By this lemma, we are able to establish the convergence rate result of Algorithm \ref{alg:asyn_bcd} for solving~\eqref{eq:main} if the problem is convex.

\begin{theorem}\label{thm:ns-rate}\textbf{\emph{(Convergence rate for the nonsmooth convex case)}} Under Assumptions \ref{assump-sol} through \ref{assump2}, let $\{\vx^k\}_{k\ge1}$ be the sequence generated from Algorithm \ref{alg:asyn_bcd} with stepsize satisfying~\eqref{ns-rate-eta} and also
\begin{equation}\label{ns-eta-rate2}
\textstyle\eta\le \Big(L_c+\frac{2 L_f \gamma^2_{\rho,2}}{m}+\frac{2 L_r\gamma_{\rho,2}}{\sqrt{m}}\Big)^{-1},
\end{equation}
where $\gamma_{\rho,1}$ and $\gamma_{\rho,2}$ are defined in~\eqref{def-gamma}. We have
\begin{enumerate}
\item
If the function $F$ is convex, then
\begin{align}\label{ineq4-thm-rate}
\textstyle \EE[F(\vx^{k})-F^*]
\le \frac{m\Phi(\vx^0)}{2\eta(m+k)},
\end{align}
where
$$\textstyle\Phi(\vx^k)=\EE\left\|\vx^{k}-\cP_{X^*}(\vx^{k})\right\|^2+2\eta\EE[F(\vx^{k})-F^*].$$

\item If $F$ is strongly convex with constant $\mu$, 
then
\begin{align}\label{thm-linear-rate}
\textstyle\Phi(\vx^k)\le \textstyle \left(1-\frac{\eta\mu}{m(1+\eta\mu)}\right)^{k} \Phi(\vx^0).
\end{align}
\end{enumerate}
\end{theorem}

Before proving this theorem, we make two remarks and present a few lemmas below.
\begin{remark}
Similar to~\eqref{sm-linear-rate}, for the linear convergence result~\eqref{thm-linear-rate}, the strong convexity assumption can be weakened to \emph{optimal strong convexity}. The latter one is strictly weaker than the former one; see \emph{\cite{liu2015async-scd}} for more discussions.
\end{remark}

\begin{remark}[Comparison of stepsize]
For the special case that the delay is bounded by $\tau=o(\sqrt[4]{m})$, choosing $\rho =O(1+\frac{1}{\tau})$, we have both $\gamma_{\rho,1}$ and $\gamma_{\rho,2}$ are $O(\tau)$. Thus we can take stepsize almost $\frac{1}{L_c}$, which is larger than the stepsize $\frac{1}{2L_c}$ given in~\cite{liu2015async-scd}.
\end{remark}

\begin{lemma}\label{lem:thm-ns-rate-lem1}
Let $\gamma_{\rho,2}$ be defined in~\eqref{def-gamma}. We have
\begin{align}\label{ineq1-thm-rate-1}
\EE\langle\nabla_{i_k}f(\vx^{k-j_k})-\nabla_{i_k}f(\vx^{k}),\vx_{i_k}^{k}-\vx_{i_k}^{k+1}\rangle
\le  \textstyle\frac{L_r\gamma_{\rho,2}}{m\sqrt{m}}\EE \|\vd^{k}\|^2.
\end{align}
\end{lemma}
\begin{proof}
It is proved via the Cauchy-Schwarz inequality, the bound~\eqref{ineq6-2-dx}, and $\EE\langle\nabla_{i_k}f(\vx^{k-j_k})-\nabla_{i_k}f(\vx^{k}),\vx_{i_k}^{k}-\vx_{i_k}^{k+1}\rangle
= \tfrac{1}{m}\EE\langle\nabla f(\vx^{k-j_k})-\nabla f(\vx^{k}),\vd^{k}\rangle$. \qed\end{proof}
\begin{lemma}\label{lem:thm-ns-rate-lem2}
It holds that
\begin{align}\label{ineq1-thm-rate-2}
&\EE\left[f(\vx^k)-f(\vx^{k+1})+r_{i_k}((\cP_{X^*}(\vx^{k}))_{i_k})-r_{i_k}(\vx_{i_k}^{k+1})\right] \cr
=\, &\EE[F(\vx^k)-F(\vx^{k+1})]+\tfrac{1}{m}\EE[R(\cP_{X^*}(\vx^{k}))-R(\vx^k)].
\end{align}
\end{lemma}
\begin{proof}
Equation~\eqref{ineq1-thm-rate-2} is a direct consequence of 
$r_{i_k}((\cP_{X^*}(\vx^{k}))_{i_k})-r_{i_k}(\vx_{i_k}^{k+1})=r_{i_k}((\cP_{X^*}(\vx^{k}))_{i_k})-r_{i_k}(\vx_{i_k}^k)+R(\vx^{k})-R(\vx^{k+1}).$
\qed\end{proof}
\begin{lemma}\label{lem:thm-ns-rate-lem3}
Let $\gamma_{\rho,2}$ be defined in~\eqref{def-gamma}. It holds that
\begin{align}\label{ineq1-thm-rate-3}
\EE \langle\nabla_{i_k}f(\vx^{k-j_k}), (\cP_{X^*}(\vx^{k}))_{i_k}-\vx_{i_k}^{k}\rangle
\le &\tfrac{1}{m}\EE\big[f(\cP_{X^*}(\vx^{k}))-f(\vx^k)\big]\cr
&+\tfrac{L_f\gamma_{\rho,2}^2}{m^2}\EE\|\vd^k\|^2.
\end{align}
\end{lemma}
\begin{proof}
Since $i_k$ is uniformly distributed and independent of $j_k$, we have
\begin{equation}\label{eq:pf-lem:thm-ns-rate-lem3-eq1}
\EE_{i_k} \langle\nabla_{i_k}f(\vx^{k-j_k}), (\cP_{X^*}(\vx^{k}))_{i_k}-\vx_{i_k}^{k}\rangle
= \tfrac{1}{m} \langle\nabla f(\vx^{k-j_k}), \cP_{X^*}(\vx^{k})-\vx^{k}\rangle.
\end{equation}
We split the term and apply the convexity of $f$ and Lipschitz continuity of $\nabla f$ to get
\begin{align}\label{eq:pf-lem:thm-ns-rate-lem3-eq2}
&\langle\nabla f(\vx^{k-j_k}), \cP_{X^*}(\vx^{k})-\vx^{k}\rangle \cr
=&\langle\nabla f(\vx^{k-j_k}), \cP_{X^*}(\vx^{k})-\vx^{k-j_k}\rangle \cr
 & +\langle\nabla f(\vx^k) + \nabla f(\vx^{k-j_k})-\nabla f(\vx^k), \vx^{k-j_k}-\vx^{k}\rangle\cr
\le &\big[f(\cP_{X^*}(\vx^{k}))-f(\vx^{k-j_k})+f(\vx^{k-j_k})-f(\vx^k)\big] \cr
    &+\langle\nabla f(\vx^{k-j_k})-\nabla f(\vx^k), \vx^{k-j_k}-\vx^{k}\rangle\cr
\le &\big[f(\cP_{X^*}(\vx^{k}))-f(\vx^k)\big]+ L_f\|\vx^{k-j_k}-\vx^{k}\|^2.
\end{align}
Substituting~\eqref{eq:pf-lem:thm-ns-rate-lem3-eq2} into~\eqref{eq:pf-lem:thm-ns-rate-lem3-eq1} and taking expectation yield
\begin{align*}
&\EE \langle\nabla_{i_k}f(\vx^{k-j_k}), (\cP_{X^*}(\vx^{k}))_{i_k}-\vx_{i_k}^{k}\rangle\\
\le &\tfrac{1}{m}\EE\big[f(\cP_{X^*}(\vx^{k}))-f(\vx^k)\big]+\tfrac{L_f}{m}\EE\|\vx^{k-j_k}-\vx^{k}\|^2.
\end{align*}
Noting $\|\vx^k-\vx^{k-j_k}\|^2\le j_k\sum_{d=k-j_k}^{k-1}\|\vx^{d+1}-\vx^d\|^2$, applying~\eqref{diff-zero} and~\eqref{dec-xbar} and using the definition of $\gamma_{\rho,2}$, we complete the proof of~\eqref{ineq1-thm-rate-3}.
\qed\end{proof}
\begin{lemma}\label{lem:thm-ns-rate-lem4}
Under the assumptions of Theorem \ref{thm:ns-rate}, we have $\EE [F(\vx^{k+1})]\le \EE [F(\vx^k)],\,\forall k$.
\end{lemma}
\begin{proof} Taking expectation on both side of~\eqref{ns-ineq1} and using~\eqref{ineq1-thm-rate-1} yield
\begin{align*}
\EE [F(\vx^{k+1})]  
\le&\textstyle\EE [F(\vx^k)]+\frac{1}{m}\left(\frac{L_c}{2}-\frac{1}{\eta}+{L_r\gamma_{\rho,2}\over \sqrt{m}}\right)\EE\|\vd^k\|^2,
\end{align*}
which implies $\EE [F(\vx^{k+1})]\le \EE [F(\vx^k)]$ from the condition on $\eta$ in~\eqref{ns-eta-rate2}.
\qed\end{proof}
Now we are ready to prove Theorem \ref{thm:ns-rate}.
\begin{proof}[of Theorem \ref{thm:ns-rate}]
From the update of $\vx^{k+1}$, we have
$$\vzero\in\nabla_{i_k}f(\vx^{k-j_k})+\tfrac{1}{\eta}(\vx_{i_k}^{k+1}-\vx_{i_k}^k)+\partial r_{i_k}(\vx_{i_k}^{k+1}),$$
and thus for any $\vx_{i_k}$, it holds from the convexity of $r_{i_k}$ that
\begin{equation}\label{ineq-rik}
r_{i_k}(\vx_{i_k})\ge r_{i_k}(\vx_{i_k}^{k+1})-\big\langle\nabla_{i_k}f(\vx^{k-j_k})+\tfrac{1}{\eta}(\vx_{i_k}^{k+1}-\vx_{i_k}^k), \vx_{i_k}-\vx_{i_k}^{k+1}\big\rangle.
\end{equation}
Since $\vx^{k+1}=\vx^{k}+U_{i_k}(\vx^{k+1}-\vx^k)$, we have
\begin{align}\label{eq:xkcpX}
\big\|\vx^{k+1}-\cP_{X^*}(\vx^{k})\big\|^2
=  &\big\|\vx^k-\cP_{X^*}(\vx^{k})\big\|^2-\|\vx_{i_k}^{k+1}-\vx_{i_k}^k\|^2\cr
   &+2\langle\vx_{i_k}^{k+1}-(\cP_{X^*}(\vx^{k}))_{i_k}, \vx_{i_k}^{k+1}-\vx_{i_k}^k\rangle.
\end{align}
From the definition of $\cP_{X^*}$, it follows that $\|\vx^{k+1}-\cP_{X^*}(\vx^{k+1})\|^2\le \|\vx^{k+1}-\cP_{X^*}(\vx^{k})\|^2$. Then using~\eqref{ineq-rik} and~\eqref{eq:xkcpX}, we have
\begin{align}\label{eq:xkcpX2}
\|\vx^{k+1}-\cP_{X^*}(\vx^{k+1})\|^2
\le &\|\vx^k-\cP_{X^*}(\vx^{k})\|^2-\|\vx_{i_k}^{k+1}-\vx_{i_k}^k\|^2\cr
    &+2\eta\left(r_{i_k}((\cP_{X^*}(\vx^{k}))_{i_k})-r_{i_k}(\vx_{i_k}^{k+1})\right)\cr
		&+\langle\nabla_{i_k}f(\vx^{k-j_k}), (\cP_{X^*}(\vx^{k}))_{i_k}-\vx_{i_k}^{k+1}\rangle.
\end{align}
We split the cross term to have
\begin{align*}
\langle\nabla_{i_k}f(\vx^{k-j_k}), (\cP_{X^*}(\vx^{k}))_{i_k}-\vx_{i_k}^{k+1}\rangle = \langle\nabla_{i_k}f(\vx^{k-j_k}), (\cP_{X^*}(\vx^{k}))_{i_k}-\vx_{i_k}^{k}\rangle\cr
\qquad + \langle\nabla_{i_k}f(\vx^{k}), \vx_{i_k}^{k}-\vx_{i_k}^{k+1}\rangle  +\langle\nabla_{i_k}f(\vx^{k-j_k})-\nabla_{i_k}f(\vx^{k}),\vx_{i_k}^{k}-\vx_{i_k}^{k+1}\rangle.
\end{align*}
From~\eqref{lip-ineq}, it follows that
\begin{align*}
\langle\nabla_{i_k}f(\vx^{k}), \vx_{i_k}^{k}-\vx_{i_k}^{k+1}\rangle\le f(\vx^k)-f(\vx^{k+1})+\tfrac{L_c}{2}\|\vx_{i_k}^{k}-\vx_{i_k}^{k+1}\|^2.
\end{align*}
Plugging the above two equations into~\eqref{eq:xkcpX2} gives
\begin{align}\label{ineq1-thm-rate}
   &\textstyle\left\|\vx^{k+1}-\cP_{X^*}(\vx^{k+1})\right\|^2\cr
\le &\left\|\vx^k-\cP_{X^*}(\vx^{k})\right\|^2-(1-\eta L_c)\|\vx_{i_k}^{k+1}-\vx_{i_k}^k\|^2\cr
    &+2\eta\langle\nabla_{i_k}f(\vx^{k-j_k}), (\cP_{X^*}(\vx^{k}))_{i_k}-\vx_{i_k}^{k}\rangle\cr
    &+2\eta\langle\nabla_{i_k}f(\vx^{k-j_k})-\nabla_{i_k}f(\vx^{k}),\vx_{i_k}^{k}-\vx_{i_k}^{k+1}\rangle\cr
    &+2\eta\left[f(\vx^k)-f(\vx^{k+1})+r_{i_k}((\cP_{X^*}(\vx^{k}))_{i_k})-r_{i_k}(\vx_{i_k}^{k+1})\right].
\end{align}
Substituting~\eqref{ineq1-thm-rate-1} through~\eqref{ineq1-thm-rate-3} into~\eqref{ineq1-thm-rate} and rearranging terms yield
\begin{align*}
&\textstyle\EE\left\|\vx^{k+1}-\cP_{X^*}(\vx^{k+1})\right\|^2\cr
\le &\textstyle \EE\left\|\vx^k-\cP_{X^*}(\vx^{k})\right\|^2-\frac{1}{m}\left[1-\eta L_c-\frac{2\eta L_f \gamma_{\rho,2}^2}{m}-\frac{2\eta L_r\gamma_{\rho,2}}{\sqrt{m}}\right]\EE\|\vd^k\|^2\cr
&\textstyle+\tfrac{2\eta}{m}\EE[F^*-F(\vx^k)]+2\eta\EE[F(\vx^k)-F(\vx^{k+1})]
\end{align*}
The above inequality together with~\eqref{ns-eta-rate2} implies
\begin{align*}
    &\EE\|\vx^{k+1}-\cP_{X^*}(\vx^{k+1})\|^2\\
\le &\EE\|\vx^k-\cP_{X^*}(\vx^{k})\|^2+\tfrac{2\eta}{m}\EE[F^*-F(\vx^k)]+2\eta\EE[F(\vx^k)-F(\vx^{k+1})]
\end{align*}
and thus, with the monotonicity of $\EE[F(\vx^k)]$ in Lemma \ref{lem:thm-ns-rate-lem4},
\begin{align}
&\textstyle\EE\left\|\vx^{k+1}-\cP_{X^*}(\vx^{k+1})\right\|^2+2\eta\EE[F(\vx^{k+1})-F^*]\cr
\le &\textstyle \EE\left\|\vx^k-\cP_{X^*}(\vx^{k})\right\|^2+ 2\eta\EE[F(\vx^{k})-F^*]-\tfrac{2\eta}{m}\EE[F(\vx^{k})-F^*]\label{ineq2-thm-rate}\\
\le &\textstyle \left\|\vx^0-\cP_{X^*}(\vx^{0})\right\|^2+ 2\eta\EE[F(\vx^{0})-F^*]- \tfrac{2\eta}{m}\sum_{t=0}^k\EE[F(\vx^{t})-F^*]\cr
\le &\textstyle\left\|\vx^0-\cP_{X^*}(\vx^{0})\right\|^2+ 2\eta\EE[F(\vx^{0})-F^*]- \tfrac{2\eta}{m}(k+1)\EE[F(\vx^{k+1})-F^*].\label{ineq3-thm-rate}
\end{align}
Hence,~\eqref{ineq4-thm-rate} follows.

When $F$ is strongly convex with constant $\mu$, we have
$$F(\vx^k)-F^*\ge\tfrac{\mu}{2}\|\vx^k-\cP_{X^*}(\vx^{k})\|^2,$$
and thus from~\eqref{ineq2-thm-rate}, it follows that
\begin{align*}
&\textstyle\EE\left\|\vx^{k+1}-\cP_{X^*}(\vx^{k+1})\right\|^2+2\eta\EE[F(\vx^{k+1})-F^*]\cr
\le &\textstyle \EE\left\|\vx^k-\cP_{X^*}(\vx^{k})\right\|^2+ \left(2\eta-\frac{2\eta^2\mu}{m(1+\eta\mu)}\right)\EE[F(\vx^{k})-F^*]\nonumber \\
&\textstyle-\left(\frac{2\eta}{m}-\frac{2\eta^2\mu}{m(1+\eta\mu)}\right)\frac{\mu}{2}\EE\|\vx^k-\cP_{X^*}(\vx^{k})\|^2\cr
=&\textstyle\left(1-\frac{\eta\mu}{m(1+\eta\mu)}\right)\left(\EE\left\|\vx^{k}-\cP_{X^*}(\vx^{k})\right\|^2+2\eta\EE[F(\vx^{k})-F^*]\right).
\end{align*}
Therefore,~\eqref{thm-linear-rate} follows, and we complete the proof.
\hfill\qed\end{proof}

\section{Poisson distribution}\label{sec:poisson}
We can treat the asynchronous reading and writing as a queueing system.
Assume the $p+1$ processors have the same computing power (i.e., the same speed of reading and writing). At any time $k$, suppose the update to $\vx_{i_k}$ is performed by the $p_k$-th processor, which can be treated as the server with speed (or service rate) \emph{one} of reading and writing. All the other $p$ processors can be treated as customers, each with speed (or arrival rate) \emph{one}, where any update to $\vx$ from the $p$ processors can be regarded as one customer's arrival. Under this setting, from the $p_k$-th processor starts reading $\vx$ until it finishes updating $\vx_{i_k}$, there would be $p$ customers in the queue in average, namely, the delay $j_k$ follows the Poisson distribution with parameter $p$. Summarizing the above discussion, we have the following result.
\begin{claim}\label{cm:poisson-dist}
Suppose Algorithm \ref{alg:asyn_bcd} runs on a system with $p+ 1$ processors, which have the same speed of reading and writing during the iterations. Then the delay $j_k$ follows the Poisson distribution with parameter $p$, i.e., for all $k$,
\begin{equation}\label{eq:poisson-dist}
\Prob(j_k=t)=\frac{p^t e^{-p}}{t!},\,t=0,1,\ldots,
\end{equation}
which implies no delay if $p=0$.
\end{claim}

In general, if the processors have different computing power, $j_k$ would follow Poisson distribution with a parameter being the speed ratio of the other $p$ processors to the $p_k$-th one. However, in a multi-core workstation with shared memory, the processors are usually of the same style and can have the same computing ability. In the following, we assume the distribution in~\eqref{prob-delay} to be Poisson distribution with parameter $p$ and discuss the convergence results we obtained in the previous sections. First we give the values of the expected quantities we used before.
\begin{proposition}\label{prop-const}
Suppose there are $p+1$ processors and~\eqref{eq:poisson-dist} holds. Then for any $\rho>1$, we have that for all $k$,
\begin{align}\label{eq:const-prop}
&T=\EE[\vj]= p,  &&S=\EE[\vj^2]=p(p+1),\nonumber\\
&M_\rho=\EE[\rho^{\vj}]=e^{p(\rho-1)},&&N_\rho=\EE[\vj\rho^{\vj}]=\rho pe^{p(\rho-1)},\\
&\gamma_{\rho,1}=\tfrac{e^{p(\sqrt{\rho}-1)}-1}{\sqrt{\rho}-1},
&&\gamma_{\rho,2}=\Big(\tfrac{\rho p e^{p(\rho-1)}-p}{1-\rho^{-1}}\Big)^{-1}.\nonumber
\end{align}
where $\gamma_{\rho,1}$ and $\gamma_{\rho,2}$ are defined in~\eqref{def-gamma}.
\end{proposition}
The proof of this proposition is standard. 
From the  quantities in~\eqref{eq:const-prop} and the theorems we established in the previous sections, we make the following observations:
\begin{enumerate}
\item If $p=o(\sqrt{m})$, we can guarantee the convergence of Algorithm \ref{alg:asyn_bcd} for both smooth and nonsmooth problems by setting $\eta\lessapprox \frac{1}{L_c}$ (see Theorems \ref{thm:cvg} and \ref{thm:convg-nsm}), where $\lessapprox$ means ``less than but close to'';
\item If $2e^2(p+1)+p=o(\sqrt{m})$, then choosing $\rho=1+\frac{1}{p}$, we have the convergence rate of Algorithm \ref{alg:asyn_bcd} obtained in Theorem \ref{thm:rate-cvx-sm} by setting $\eta\lessapprox\frac{2}{eL_c}$. Then $D\approx\frac{\eta}{m}$ in~\eqref{def-const-D}, and thus near-linear speedup is achieved for solving convex smooth problems;
\item If $p=o(\sqrt[4]{m})$, we can guarantee the convergence rate of Algorithm \ref{alg:asyn_bcd} in Theorem \ref{thm:ns-rate} by setting $\eta\lessapprox\frac{1}{L_c}$ and thus  a near-linear speedup for convex nonsmooth problems.
\end{enumerate}

\section{Numerical experiments}
In this section, we evaluate the numerical performance of Algorithm \ref{alg:asyn_bcd} on solving two problems: the LASSO problem and the nonnegative matrix factorization (NMF). The tests were carried out on a machine with 64GB of memory and two Intel Xeon E5-2690 v2 processors (20 cores, 40 threads). All of the experiments were coded in C++ and its threading library was used for parallelization. We use the Eigen library for numerical linear algebra operations. {To measure the delay, we use an atomic variable to track the number of iterations as defined in the paper. The atomic variable will be incremented by one for each update. For each thread, the delay is calculated based on the difference of the iteration counters before and after the update.} For LASSO, two different settings were used. The first one sets the stepsize by the expected delay according to the analysis of this paper, and the other one used the maximum delay from~\cite{liu2014asynchronous,liu2015async-scd} and is dubbed as AsySCD. We compared the async-BCU to the serial BCU, which can be regarded as a special case of Algorithm \ref{alg:asyn_bcd} with the delay $j_k\equiv 0,\,\forall k$. For NMF, we set the stepszie by the expected delay and test its convergence behavior with different numbers of threads.

\subsection{Parameter settings}
According to Theorem \ref{thm:convg-nsm}, the following two stepsizes were used:\footnote{For the NMF problem, $L_c$ cannot be determined in the beginning, so instead of using a uniform $L_c$, we used the gradient Lipschitz constant adaptive to the iterate.}
\begin{subequations}\label{eq:test-eta}
\begin{align}
\mbox{This paper}:\ \eta=\tfrac{1/L_c}{1+\kappa^2p^2/(2m)}\label{test-eta1},\\
\mbox{Max delay}:\ \eta=\tfrac{1/L_c}{1+\kappa^2\tau^2/(2m)},\label{test-eta2}
\end{align}
\end{subequations}
where $\tau$ equals the maximum number of the generated sequence of delays.

\subsection{LASSO}
We measure the performance of Algorithm \ref{alg:asyn_bcd} on the LASSO problem~\cite{tibshirani1996-Lasso}
\begin{equation}\label{eq:lasso}
\Min_{\vx\in\RR^n} \tfrac{1}{2}\|\vA\vx-\vb\|_2^2+\lambda\|\vx\|_1,
\end{equation}
where $\vA\in\RR^{N\times n}$, $\vb\in\RR^N$, and $\lambda$ is a parameter balancing the fitting term and the regularization term. We randomly generated $\vA$ and $\vb$ following the standard normal distribution. The size was fixed to $n = 2N$ and $N=10,000$, and $\lambda=\frac{1}{N}$ was used. The Lipschitz constant $L_c = \max\{ \|(\vA_i^T\vA_i)\|^2,  ~\forall i \}$, where $\vA_i$ represents the $i$th column block of $\vA$.

Figure~\ref{fig:lasso_delay_distribution} shows the delay distribution of Algorithm~\ref{alg:asyn_bcd} with different numbers of threads. The blue bars are the normalized histogram so that the bar heights add to 1. Orange curve is the probability density function of Poisson distribution. By using 5 and 10 threads, we observe that the number of delays is concentrated on 4, and 9 respectively. When the number of threads is relatively large, the actual delay distribution closely matches with the theoretical distribution as we discussed in Section~\ref{sec:poisson}. For 20 threads, an interesting observation is that, the actual probability density is higher than the theoretical probability density when the number of delays is around 9. We think this is due to the architecture of the testing environment, i.e., the average delay within a CPU is smaller than the average delay across two different CPUs. We observe a similar behavior when 40 threads are used.

Figure~\ref{fig:lasso} plots the convergence behavior of Algorithm \ref{alg:asyn_bcd} running on 40 threads with different block sizes. We partition $\vx$ into $m$ equal-sized blocks with block sizes varying among $\{10,~50,~100,~500\}$. The results of the serial randomized coordinate descent method is also plotted for comparison. Here, one epoch is equivalent to updating all coordinates once.  {Comparing to the serial method, we observe that the delay does affect the convergence speed, and the affect becomes weaker as $m$ increases. Hence, Algorithm \ref{alg:asyn_bcd} can have nearly linear speed-up when the number of blocks is large}. In addition, we note that the stepsize setting of AsySCD is too conservative, and Algorithm \ref{alg:asyn_bcd} with stepsize set by the expected delay converges significantly faster. However, we observed that, in general, we could not take larger stepsize than that in~\eqref{test-eta1}. Some divergence behaviors are observed when using stepsizes larger than that in~\eqref{test-eta1}.

\begin{figure}[!ht]
{\small\begin{tabular}{cc}
5 threads & 10 threads\\[-0.1cm]
\includegraphics[width=0.44\textwidth,height=0.4\textwidth]{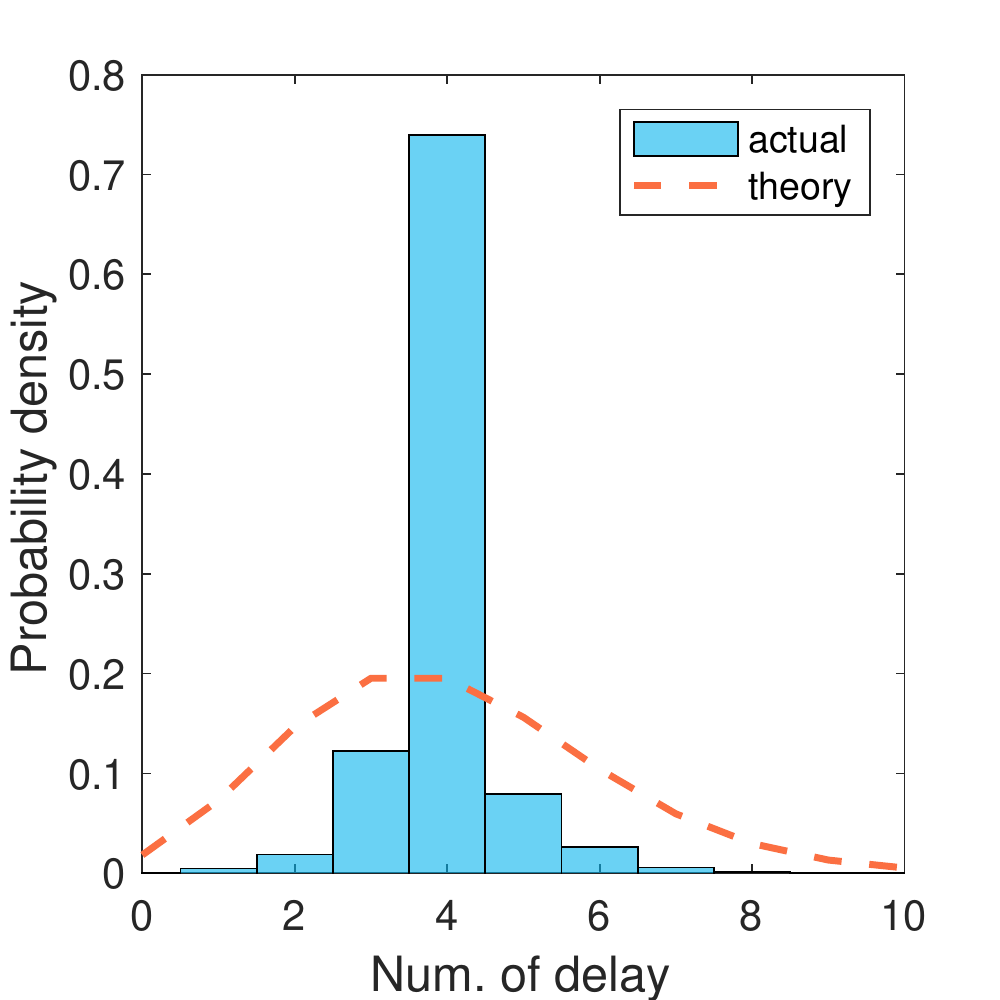}&
\includegraphics[width=0.44\textwidth,height=0.4\textwidth]{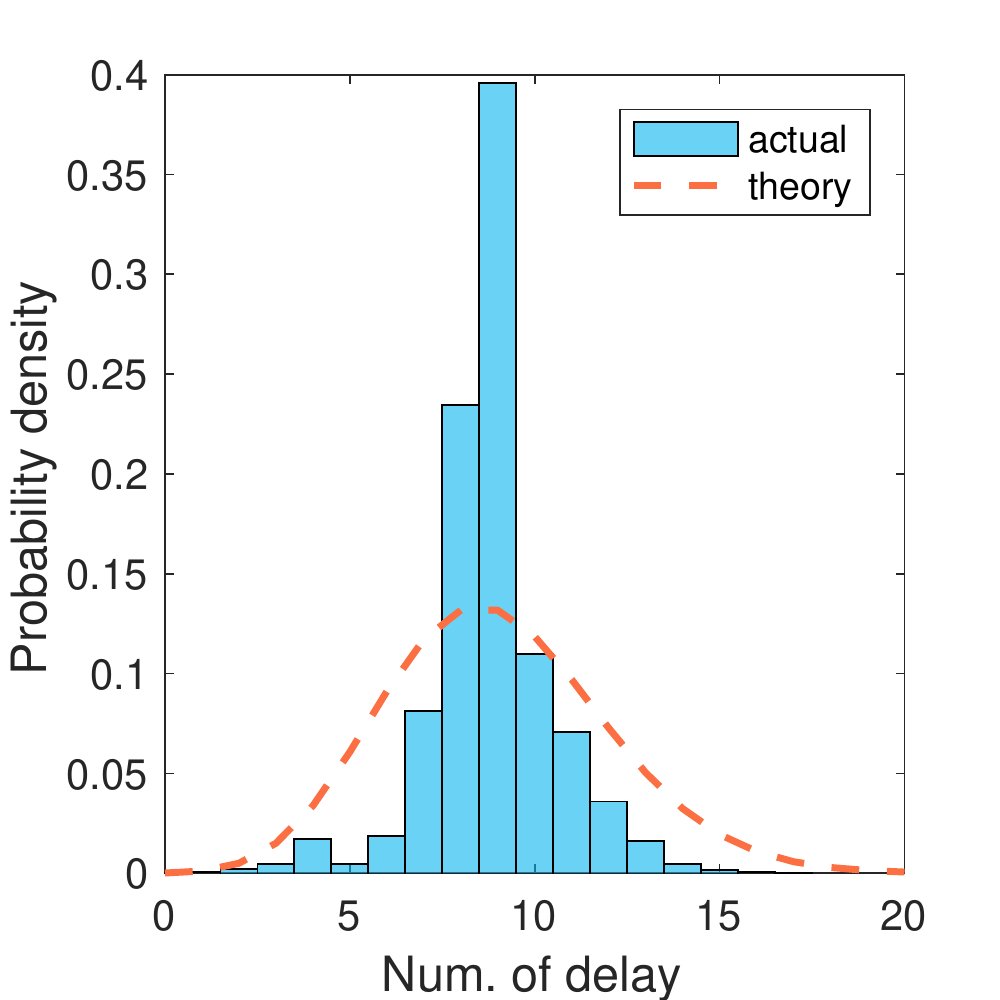}\\[0.1cm]
20 threads & 40 threads\\[-0.08cm]
\includegraphics[width=0.44\textwidth,height=0.4\textwidth]{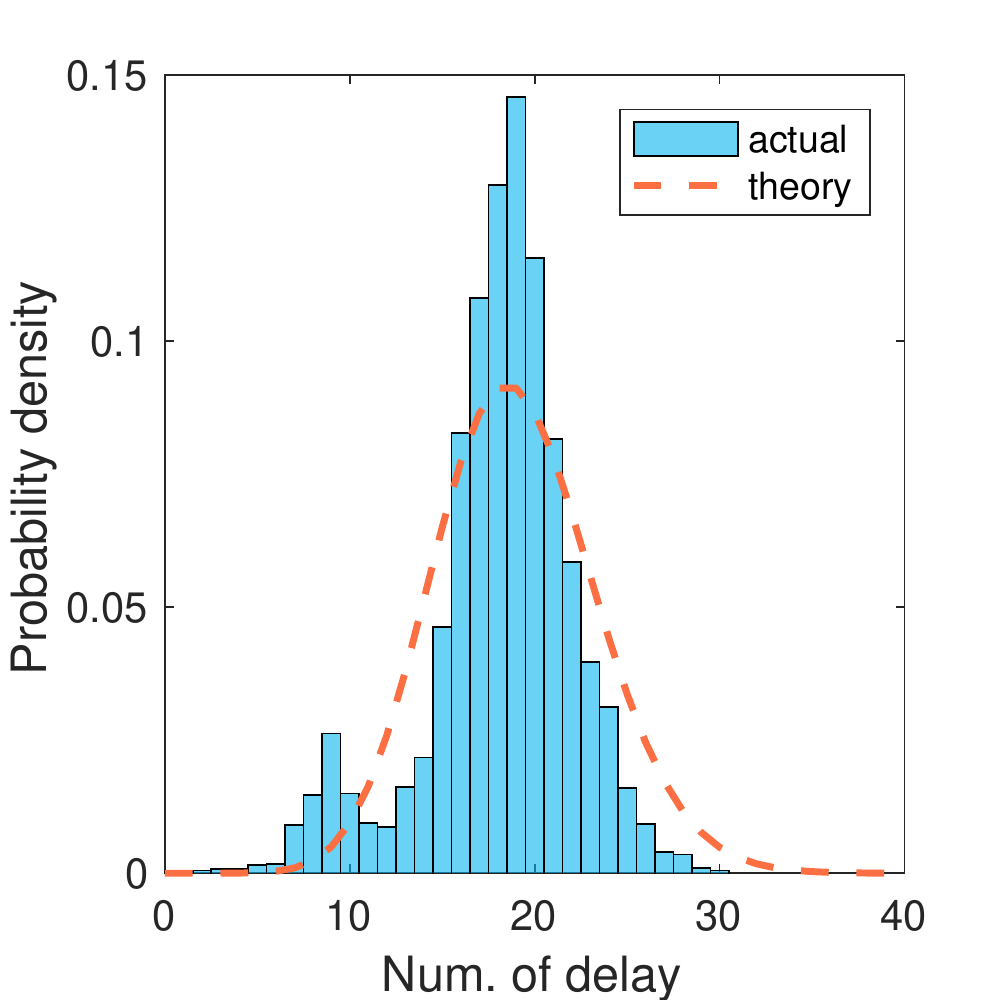}&
\includegraphics[width=0.44\textwidth,height=0.4\textwidth]{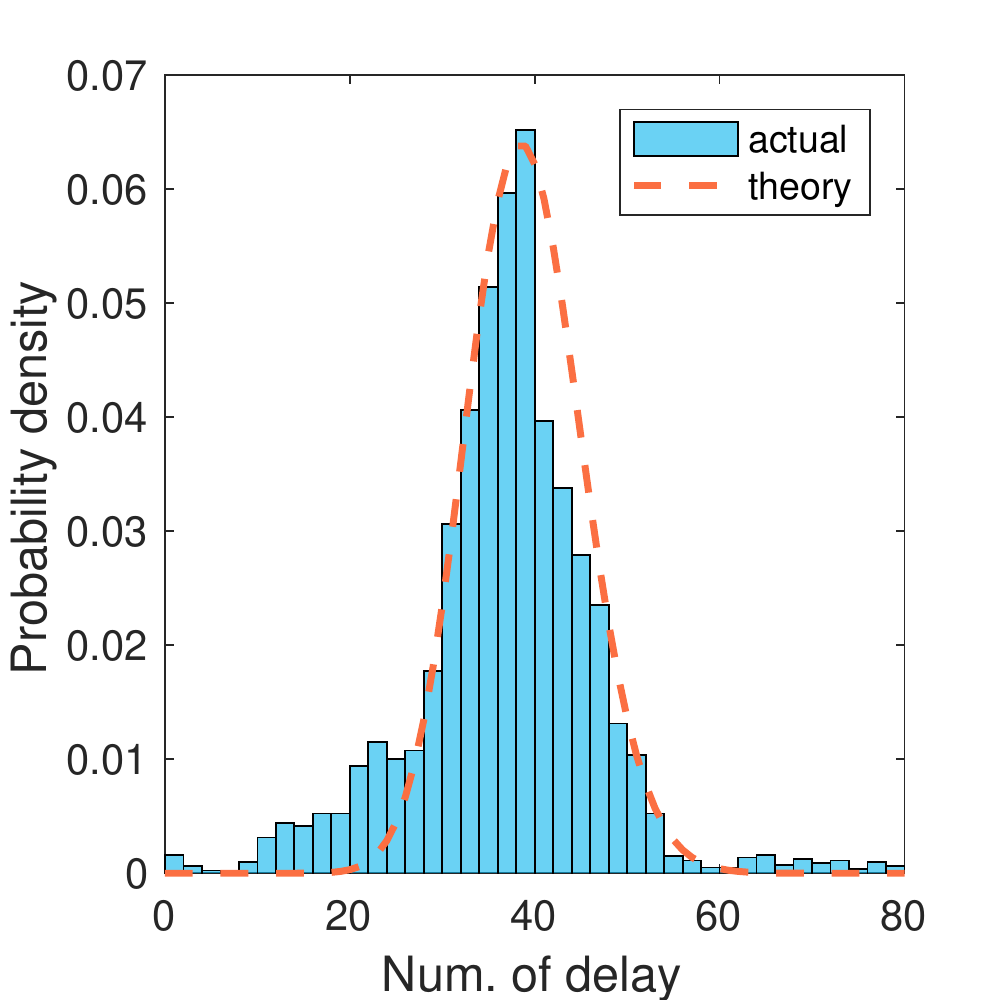}
\end{tabular}}
\caption{Delay distribution behaviors of Algorithm \ref{alg:asyn_bcd} for solving LASSO~\eqref{eq:lasso}. The tested problem has $20,000$ coordinates, and it was running with 5, 10, 20, and 40 threads.}
\label{fig:lasso_delay_distribution}
\end{figure}

\begin{figure}[!ht]
{\small \begin{tabular}{cccc}
2,000 blocks & 400 blocks & 200 blocks & 40 blocks\\[-0.08cm]
\includegraphics[width=0.22\textwidth,height=0.2\textwidth]{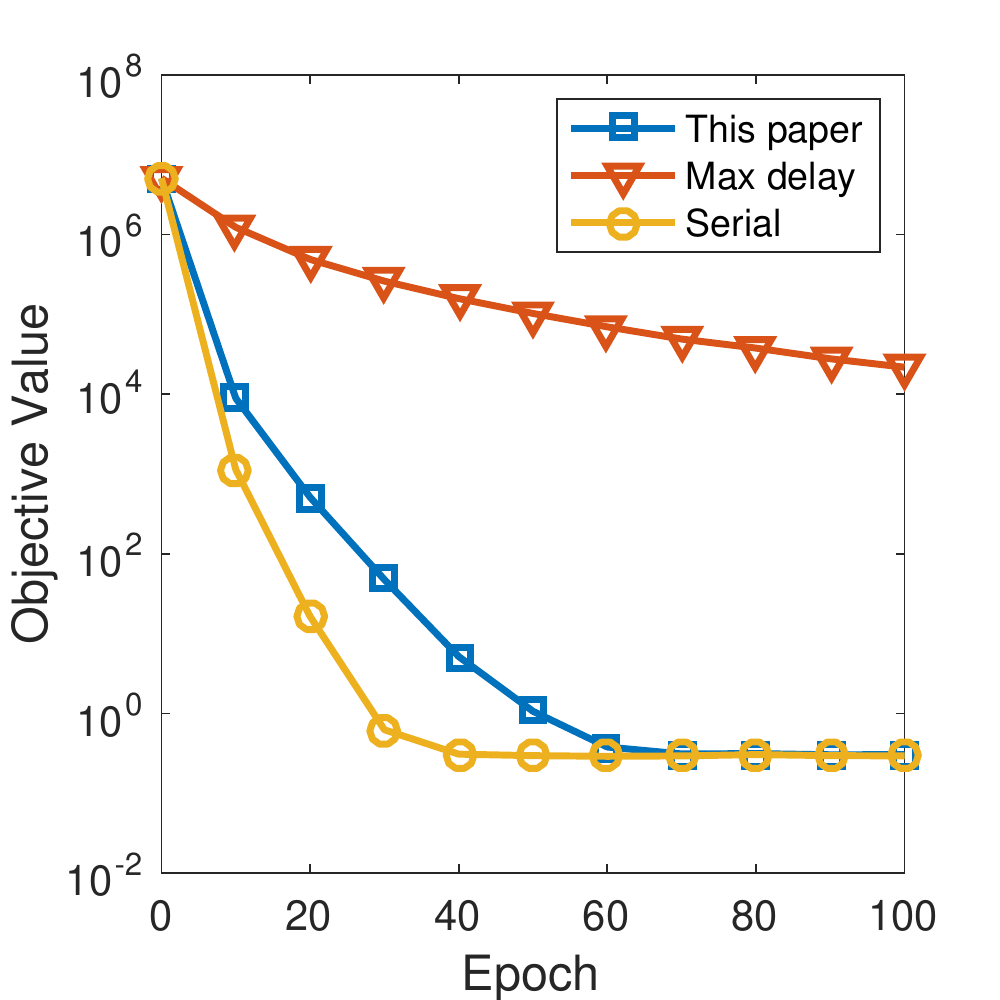}&
\includegraphics[width=0.22\textwidth,height=0.2\textwidth]{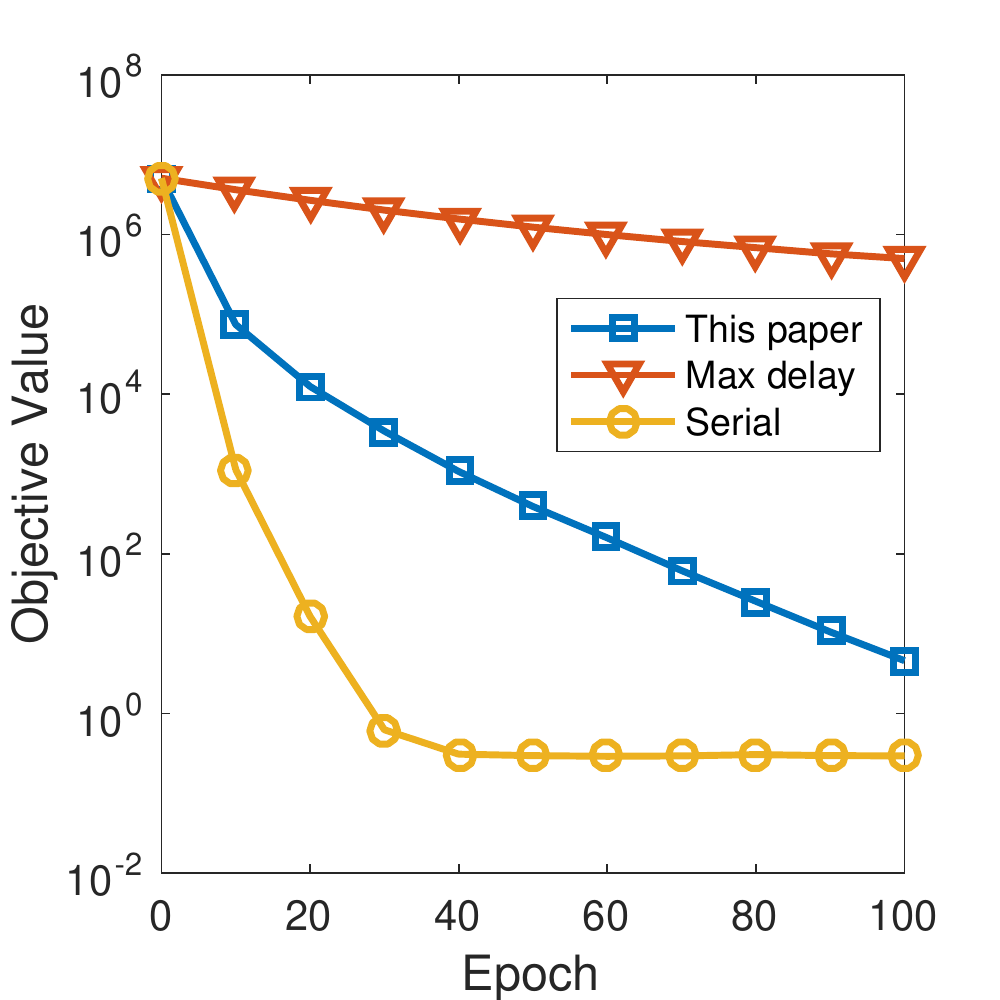} &
\includegraphics[width=0.22\textwidth,height=0.2\textwidth]{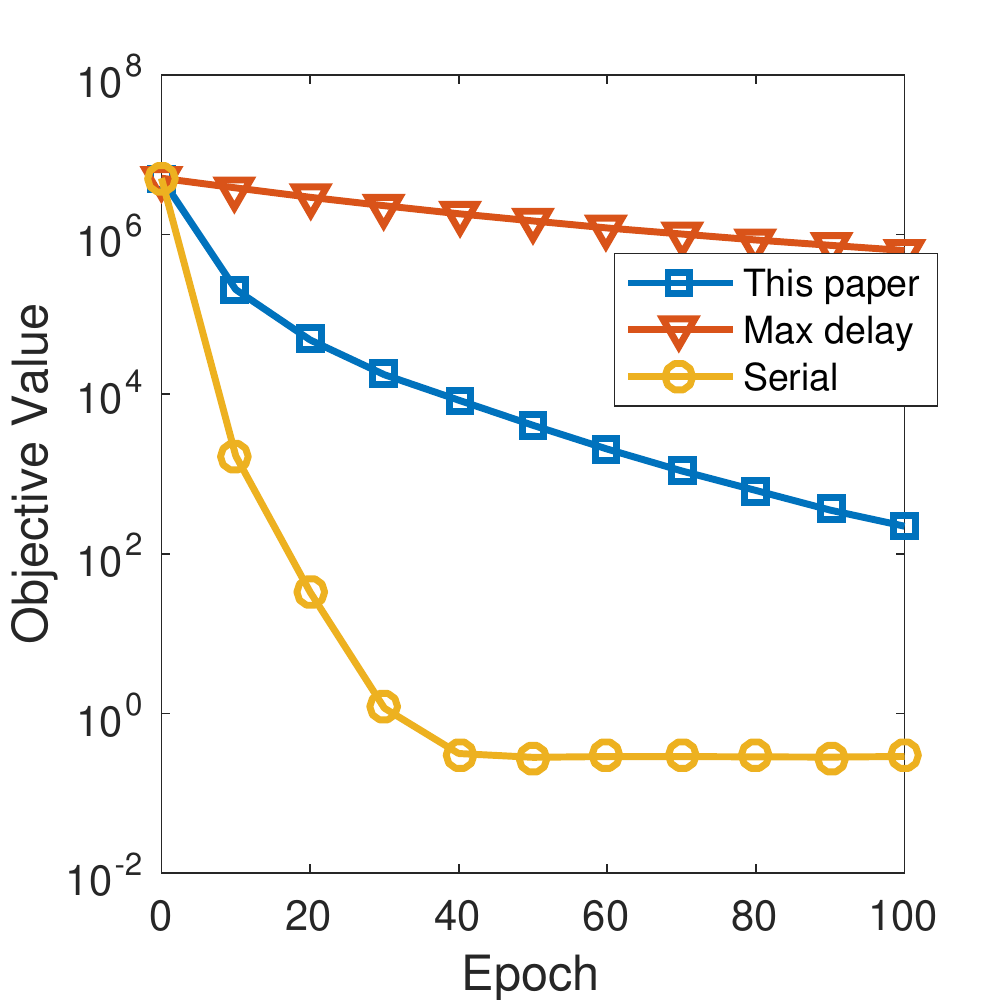}&
\includegraphics[width=0.22\textwidth,height=0.2\textwidth]{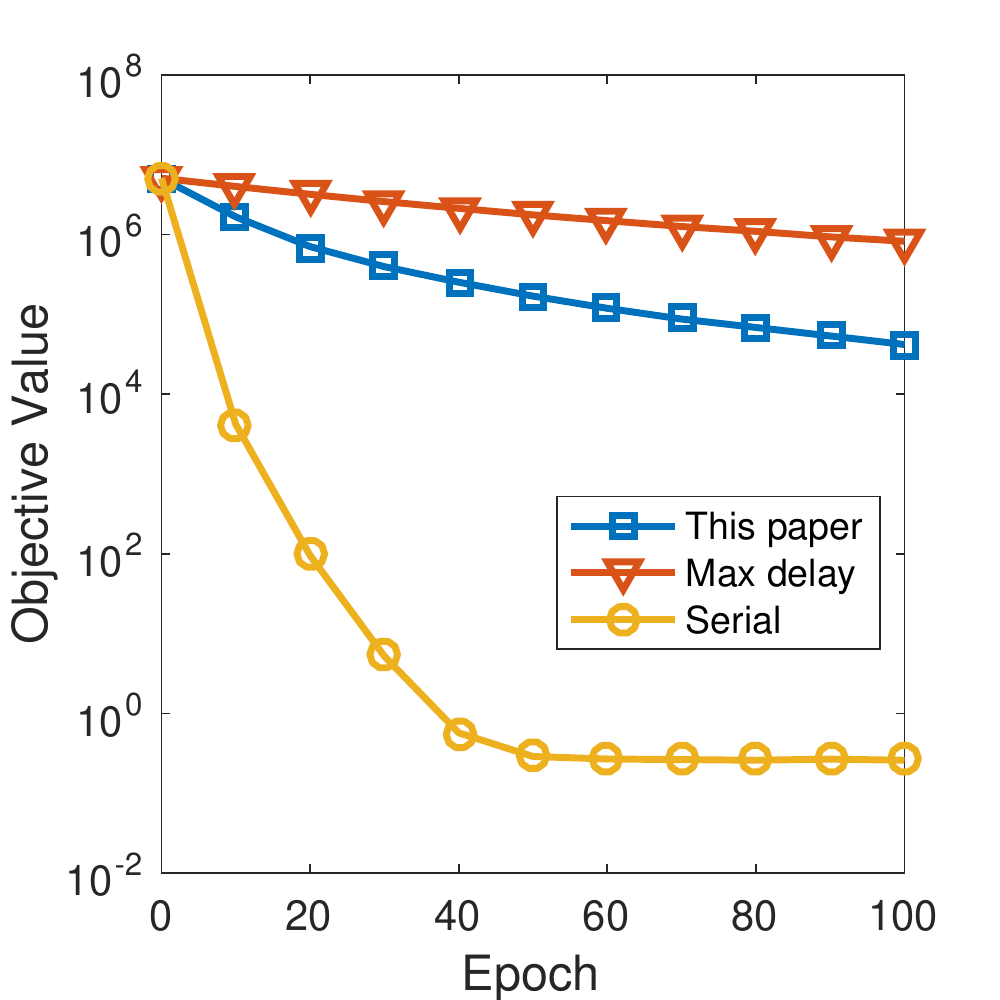}\\
\includegraphics[width=0.22\textwidth,height=0.2\textwidth]{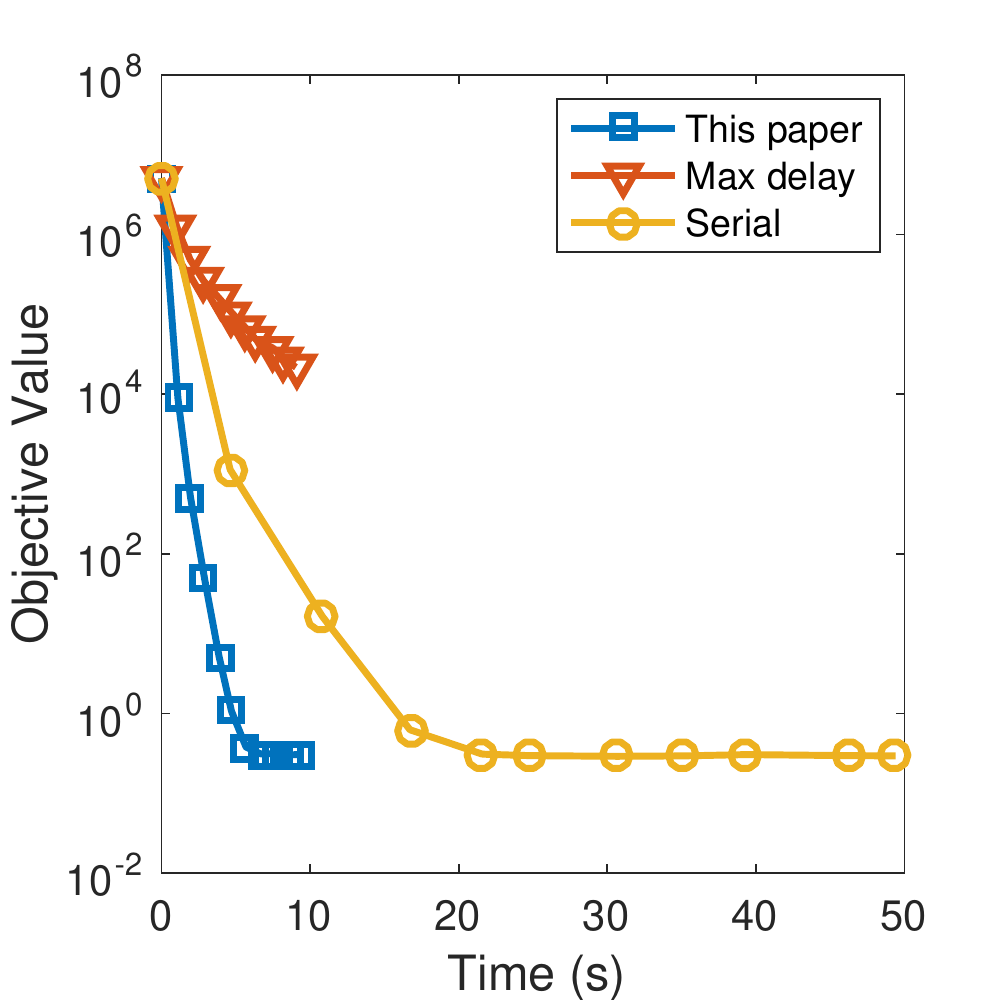}&
\includegraphics[width=0.22\textwidth,height=0.2\textwidth]{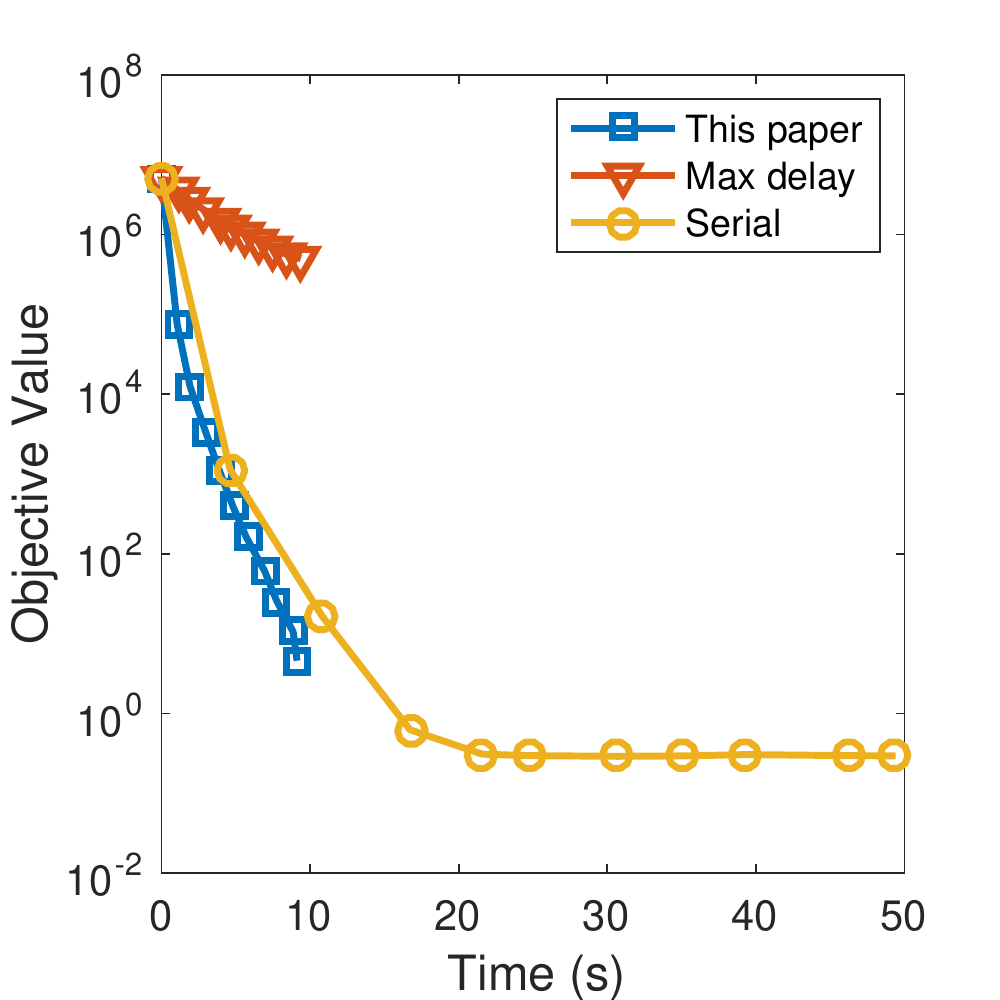} &
\includegraphics[width=0.22\textwidth,height=0.2\textwidth]{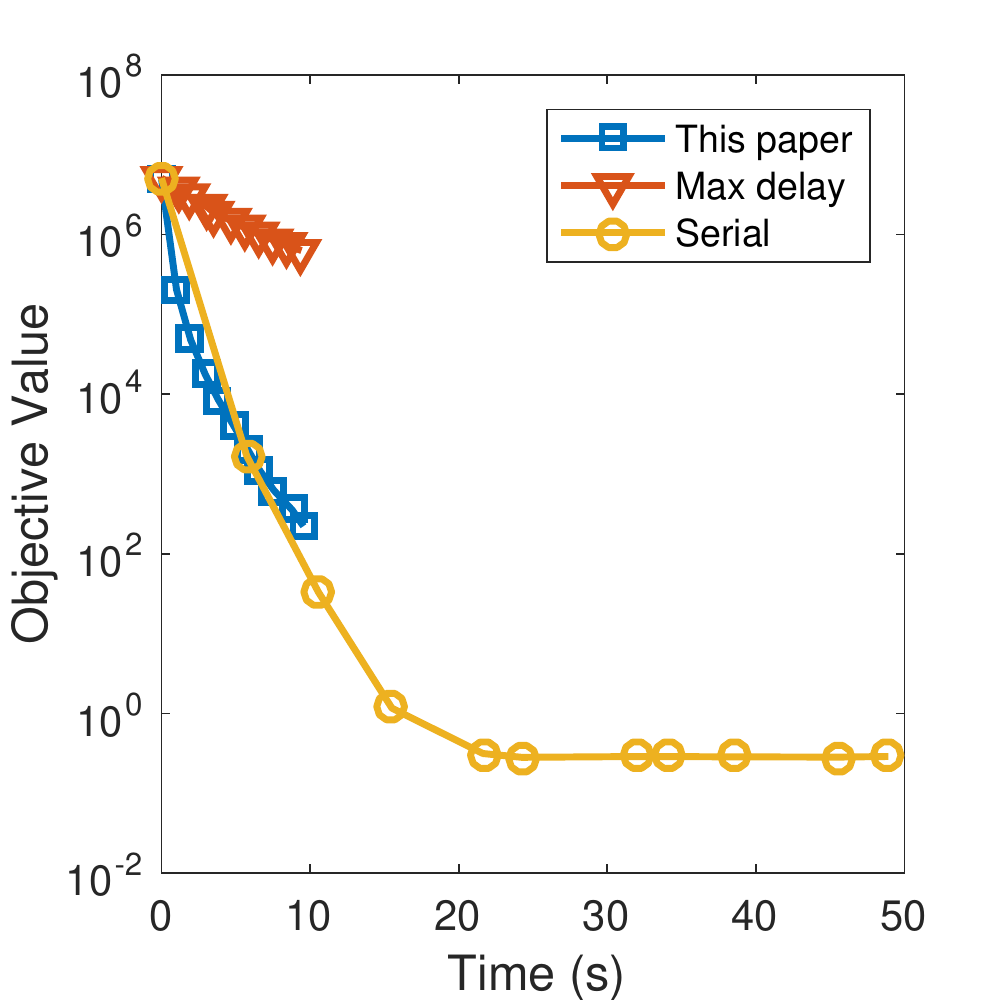}&
\includegraphics[width=0.22\textwidth,height=0.2\textwidth]{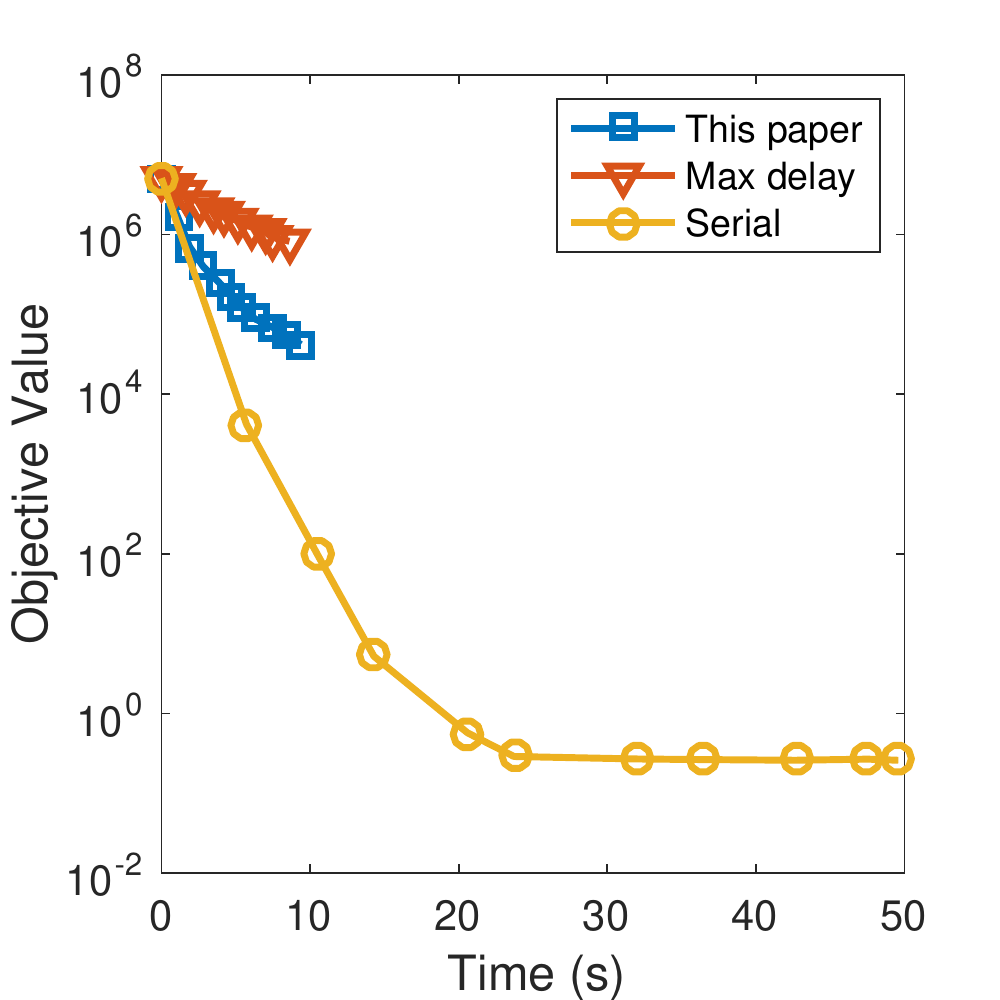}\\[0.1cm]
\end{tabular}}
\caption{ {Convergence behaviors of Algorithm \ref{alg:asyn_bcd} for solving the LASSO problem~\eqref{eq:lasso} with the stepsize given in~\eqref{eq:test-eta}, and also the serial randomized coordinate descent method. The tested problem has $10,000$ samples and $20,000$ coordinates that are evenly partitioned into $m$ blocks. It was simulated as running with 40 threads. We run 100 epochs for each experiments.}}
\label{fig:lasso}
\end{figure}

\subsection{Nonnegative matrix factorization (NMF)}
This section presents the numerical results of applying Algorithm \ref{alg:asyn_bcd} for solving the NMF problem~\cite{paatero1994-NMF}
\begin{equation}\label{eq:nmf}
\begin{array}{l}
\Min\limits_{\vX,\vY}\ \frac{1}{2}\|\vX\vY^\top-\vZ\|_F^2, \\[0.2cm]
\st \vX\in\RR^{M\times m}_+,\vY\in\RR^{N\times m}_+,
\end{array}
\end{equation}
where $\vZ\in\RR^{M\times N}_+$ is a given nonnegative matrix. We generated $\vZ=\vZ_L\vZ_R^\top$ with the elements of $\vZ_L$ and $\vZ_R$ first drawn from the standard normal distribution and then projected into the nonnegative orthant. The size was fixed to $M=N=10,000$ and $m=100$.

We treated one column of $\vX$ or $\vY$ as one block coordinate, and during the iterations, every column of $\vX$ was kept with unit norm. Therefore, the partial gradient Lipschitz constant equals \emph{one} if one column of $\vY$ is selected to update and $\|\vy_{i_k}^k\|_2^2$ if the $i_k$-th column of $\vX$ is selected. Since $\|\vy_{i_k}^k\|_2^2$ could approach to \emph{zero}, we set the Lipschitz constant to $\max(0.001,  \|\vy_{i_k}^k\|_2^2)$. This modification can guarantee the whole sequence convergence of the coordinate descent method~\cite{xu2014-ecyc}. Due to nonconvexity, global optimality cannot be guaranteed. Thus, we set the starting point close to $\vZ_L$ and $\vZ_R$. Specifically, we let $\vX^0=\vZ_L+0.5\bm{\Xi}_L$ and $\vY^0=\vZ_R+0.5\bm{\Xi}_R$ with the elements of $\bm{\Xi}_L$ and $\bm{\Xi}_R$ following the standard normal distribution. All methods used the same starting point.


Figure \ref{fig:nmf_delay_distribution} shows the delay distribution behavior of Algorithm~\ref{alg:asyn_bcd} for solving NMF. The observation is similar to Figure~\ref{fig:lasso_delay_distribution}. Figure \ref{fig:nmf} plots the convergence results of Algorithm~\ref{alg:asyn_bcd} running with $1,~5,~10,~20$ and $40$ threads. From the results, we observe that Algorithm~\ref{alg:asyn_bcd} scales up to 10 threads for the tested problem. Degenerated convergence is observed with 20 and 40 threads. This is mostly due to the following three reasons:  (1) since the number of blocks is relatively small ($m =200$), as shown in~\eqref{test-eta1}, using more threads leads to smaller stepsize, hence, slower convergence; (2) the gradient used for the current update is more staled when a relative large number of threads are used, which also leads to slow convergence; (3) high cache miss rates and false sharing also downgrade the speedup performance.
\begin{figure}[!ht]
{\small\begin{tabular}{cc}
5 threads & 10 threads\\[-0.08cm]
\includegraphics[width=0.44\textwidth,height=0.4\textwidth]{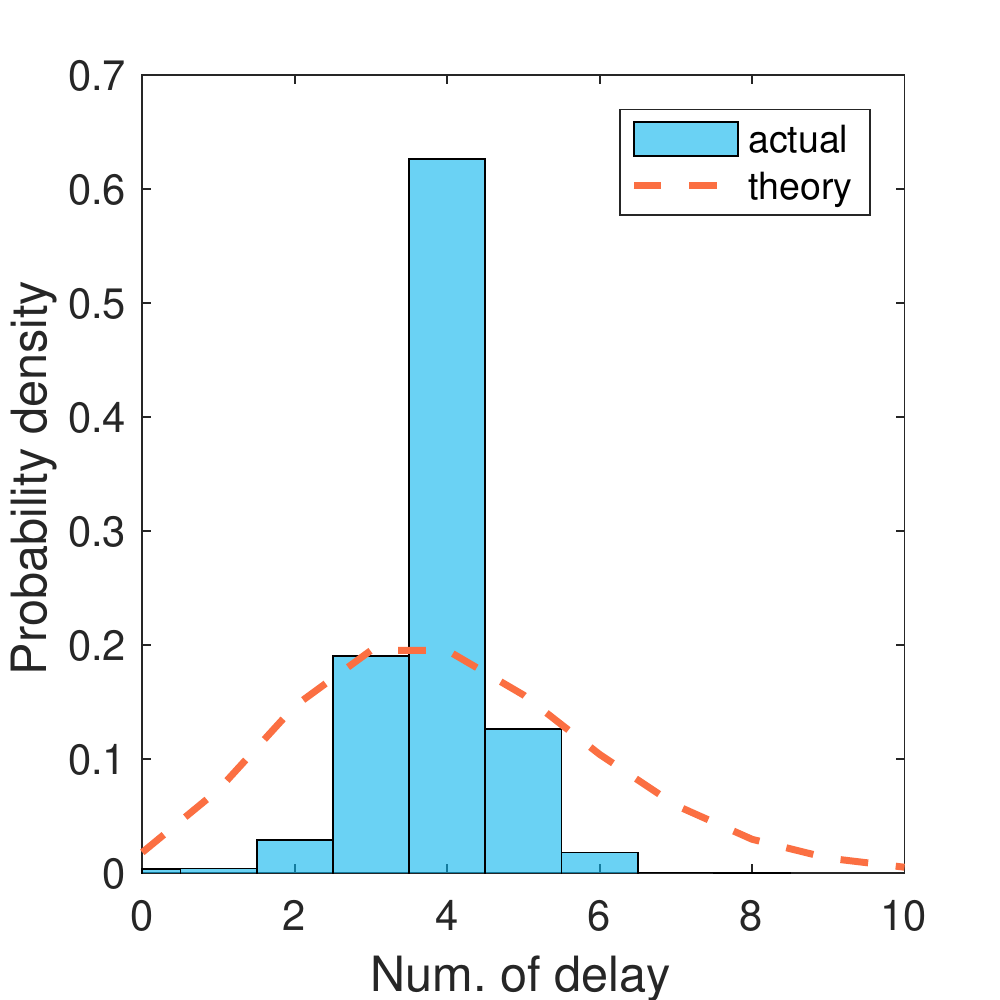}&
\includegraphics[width=0.44\textwidth,height=0.4\textwidth]{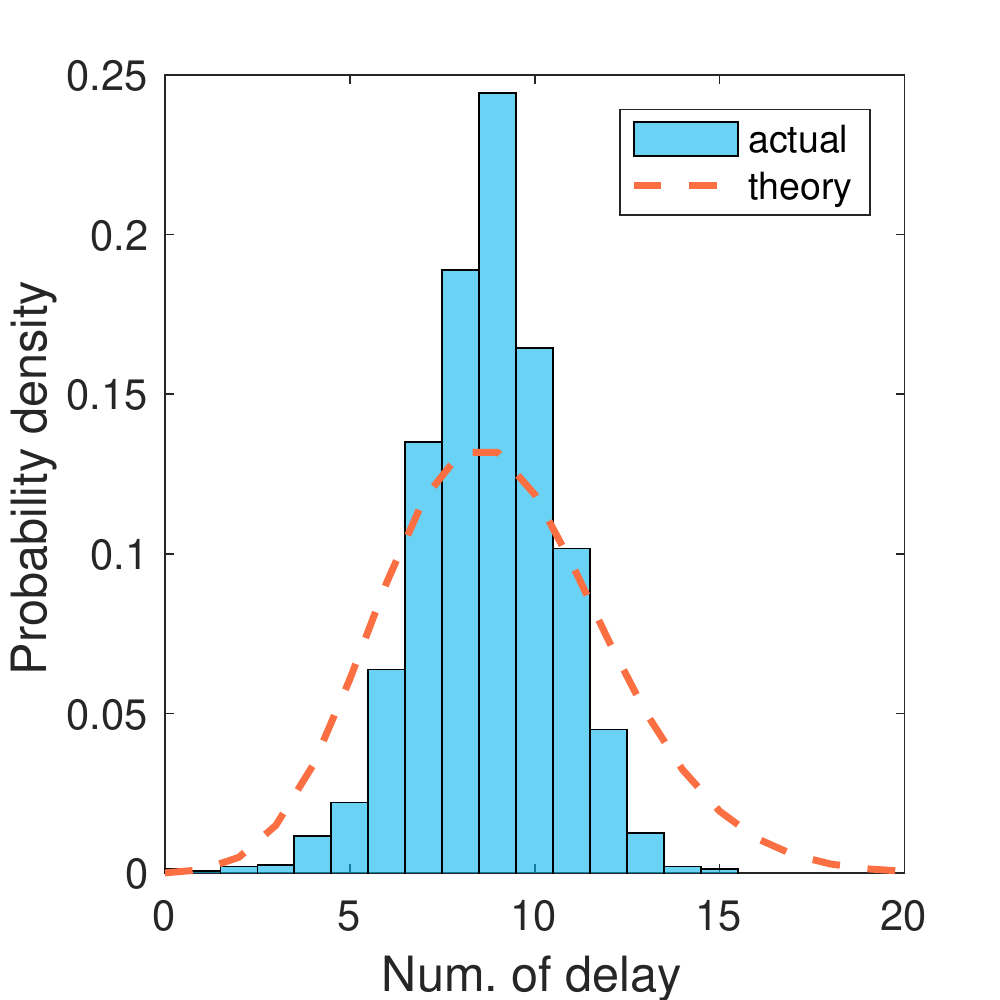}\\[0.1cm]
20 threads & 40 threads\\[-0.08cm]
\includegraphics[width=0.44\textwidth,height=0.4\textwidth]{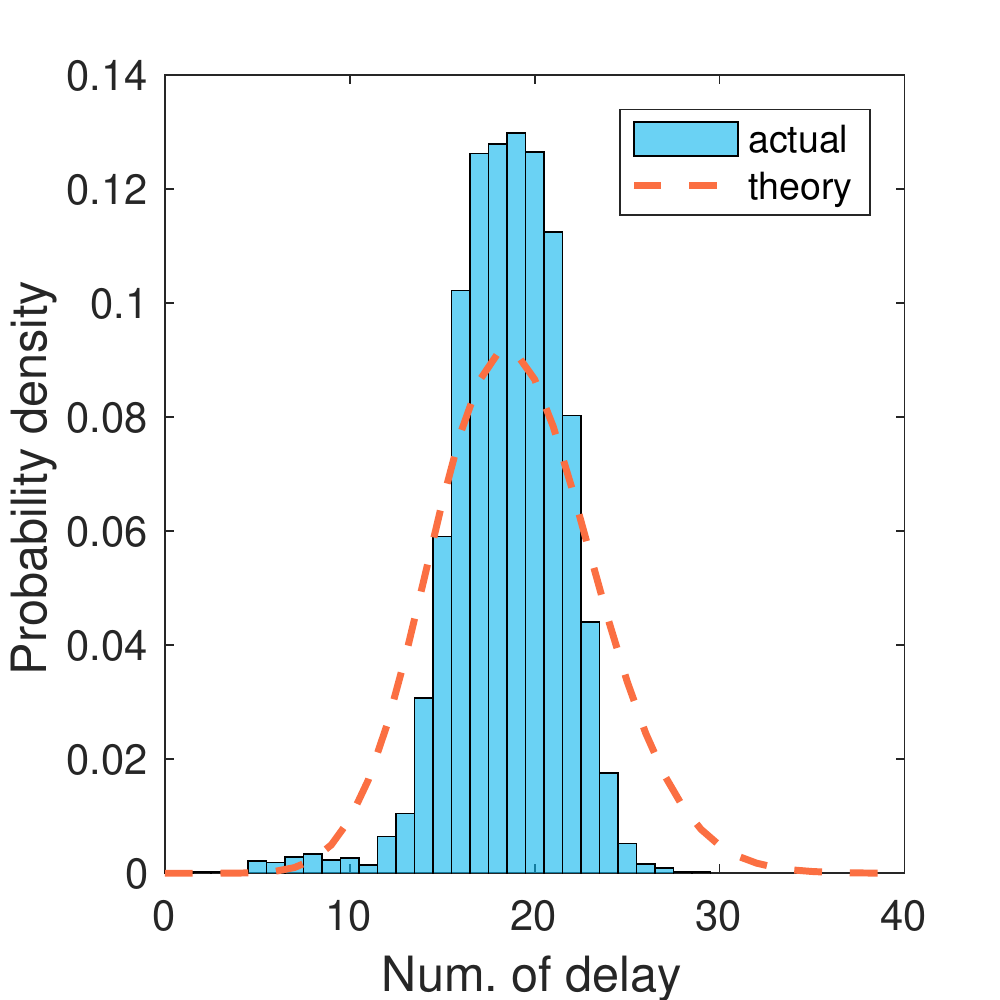}&
\includegraphics[width=0.44\textwidth,height=0.4\textwidth]{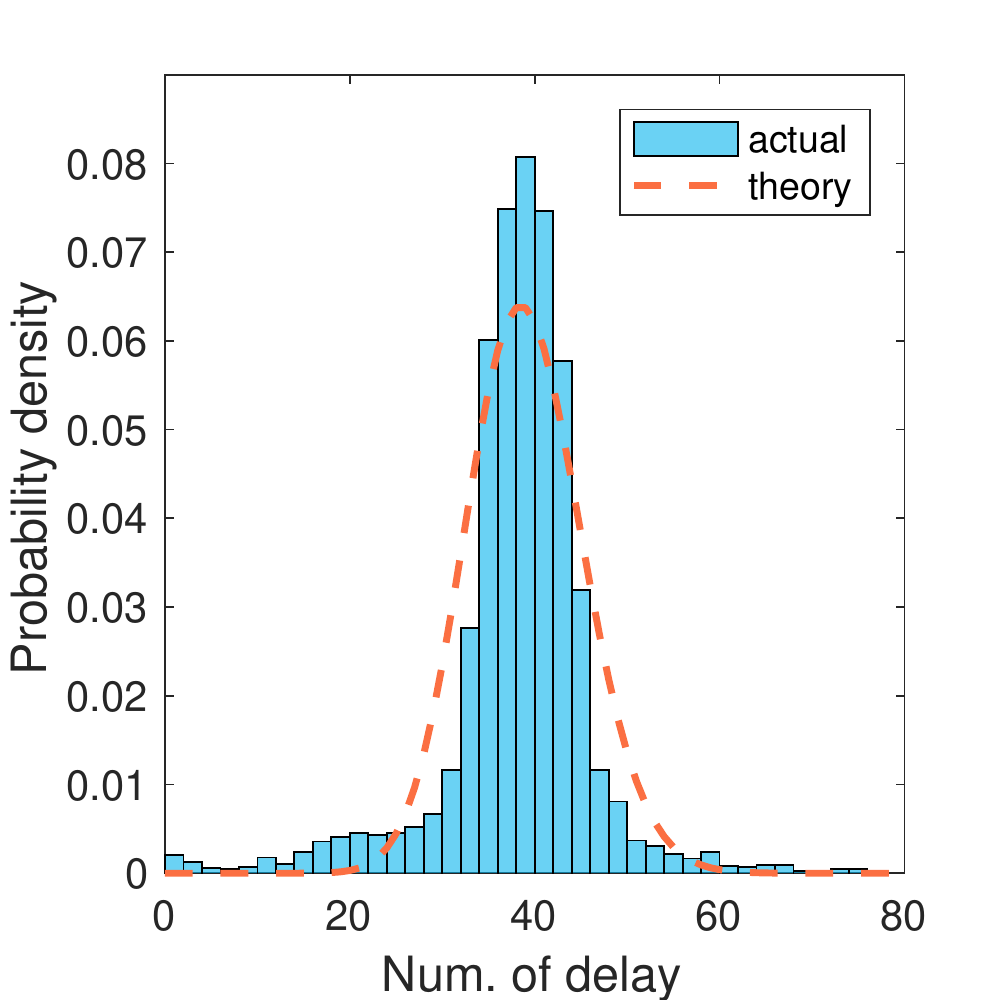}
\end{tabular}}
\caption{Delay distribution behaviors of Algorithm \ref{alg:asyn_bcd} for solving NMF~\eqref{eq:nmf}. It was running with 5, 10, 20, and 40 threads.}
\label{fig:nmf_delay_distribution}
\end{figure}
\begin{figure}[!ht]
{\small \begin{tabular}{cc}
\includegraphics[width=0.44\textwidth,height=0.4\textwidth]{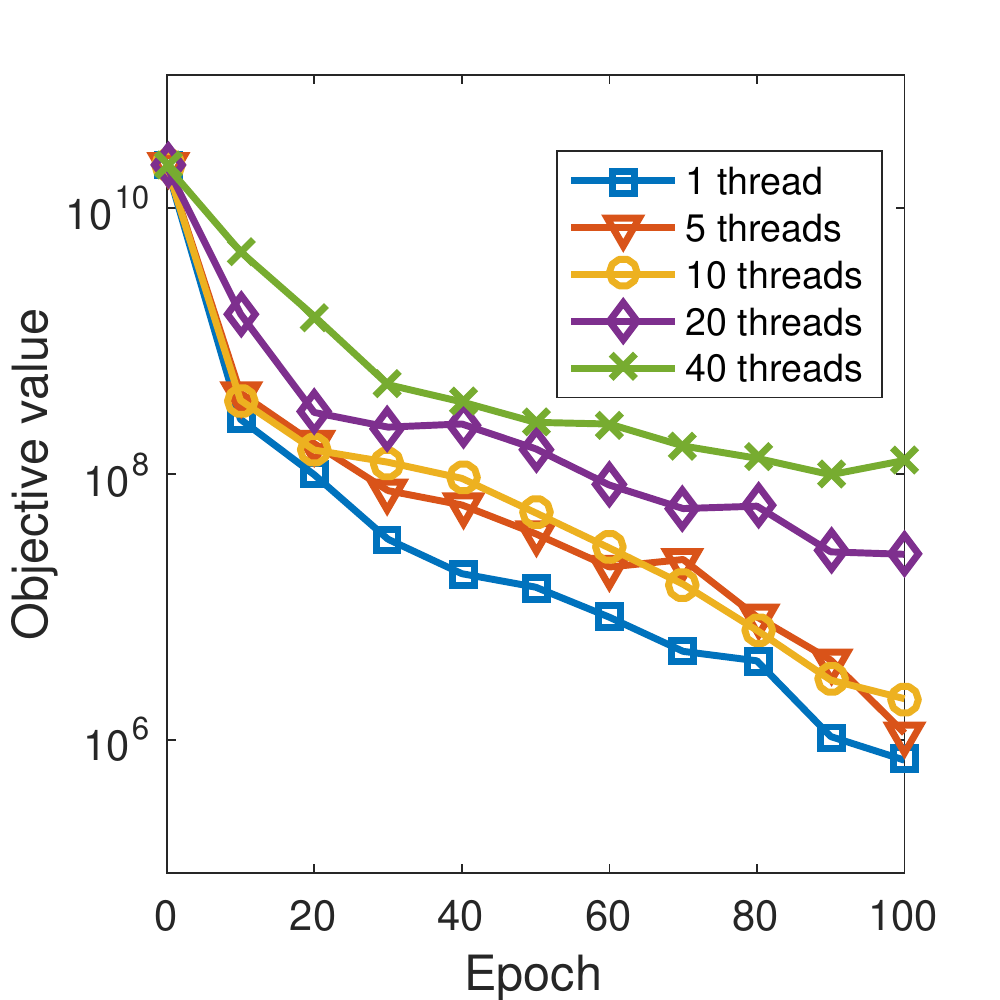}&
\includegraphics[width=0.44\textwidth,height=0.4\textwidth]{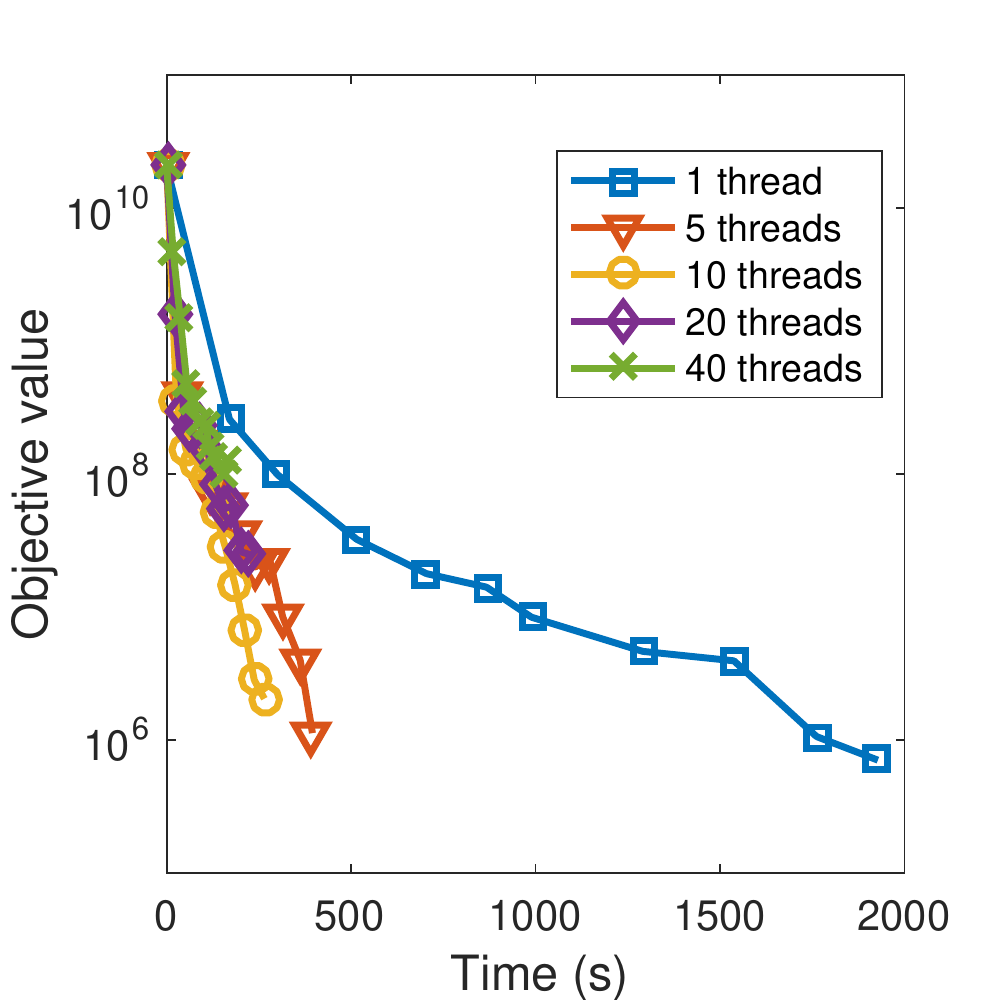}
\end{tabular}}
\caption{Convergence behaviors of Algorithm \ref{alg:asyn_bcd} for solving the NMF problem~\eqref{eq:nmf} with the stepsize set based on the expected delay. The size of the tested problem is $M=N=10,000$ and $m=100$, i.e., 200 block coordinates, and the algorithm was tested with 1, 5, 10, 20, and 40 threads.}
\label{fig:nmf}
\end{figure}

\section{Conclusions}
We have analyzed the convergence of the async-BCU method for solving both convex and nonconvex problems in a probabilistic way. We showed that the algorithm is guaranteed to converge for smooth problems if the expected delay is finite and for nonsmooth problems if the variance of the delay is also finite. In addition, we established sublinear convergence  of the method for weakly convex problems and linear convergence for strongly convex ones. The stepsize we obtained depends on certain expected quantities. Assuming the given $p+1$ processors perform identically, we showed that the delay follows a Poisson distribution with parameter $p$ and thus fully determined the stepsize. We have simulated the performance of the algorithm with our determined stepsize on solving LASSO and the nonnegative matrix factorization, and the numerical results validated our analysis.

\section*{Acknowledgements}

We would like to acknowledge support for this project
from the National Science Foundation (NSF EAGER ECCS-1462397, DMS-1621798, and DMS-1719549).


\bibliographystyle{spmpsci}
\bibliography{BGref,wyin}

\appendix

\section{Proofs of lemmas}
The following lemma is used in other proofs several times, and it is easy to  verify.
\begin{lemma}\label{lem:2seq}
For any scalar sequences $\{a_{i,j}\}$ and $\{b_i\}$, it holds that
\begin{align}\label{seq1}
\sum_{t=1}^{k-1}\sum_{d=k-t}^{k-1}a_{d,t}&=\sum_{d=1}^{k-1}\sum_{t=k-d}^{k-1}a_{d,t},\,\forall k \ge0,\\
\label{seq2}
\sum_{t=1}^k\sum_{d=0}^{t-1}a_{d,t}&=\sum_{d=0}^{k-1}\sum_{t=d+1}^ka_{d,t},\,\forall k\ge0.\\
\label{seq3}
\sum_{t=1}^k\sum_{d=1}^{t-1}a_{d,t}b_{t-d}&=\sum_{t=1}^{k-1}\Big(\sum_{d=t+1}^ka_{d-t,d}\Big)b_t,\,\forall k\ge0.
\end{align}
\end{lemma}



\subsection{Proof of Lemma~\ref{lem:bdgrad}}
\begin{proof}
Following the proof of Theorem 1 in~\cite{liu2014asynchronous}, we have
\begin{align}\label{eq:rel1}
    &\EE\big[\|\nabla f(\vx^t)\|^2-\|\nabla f(\vx^{t+1})\|^2\big]\cr
\le &2\EE\big[\|\nabla f(\vx^t)\|\cdot\|\nabla f(\vx^t)-\nabla f(\vx^{t+1})\|\big]\quad(\text{from }\|\vu\|^2-\|\vv\|^2\le 2\|\vu\|\cdot\|\vu-\vv\|)\cr
\le &2L_r\EE\big[\|\nabla f(\vx^t)\|\cdot\|\vx^t- \vx^{t+1}\|\big]=2\eta L_r\EE\big[\|\nabla f(\vx^t)\|\cdot\|U_{i_t}\nabla f(\vx^{t-j_t})\|\big]\cr
\le &\eta L_r\left(\frac{1}{\sqrt{m}}\EE\|\nabla f(\vx^t)\|^2+\sqrt{m}\EE\|U_{i_t}\nabla f(\vx^{t-j_t})\|^2\right)\cr
=   &\frac{\eta L_r}{\sqrt{m}}\left(\EE\|\nabla f(\vx^t)\|^2+\EE\|\nabla f(\vx^{t-j_t})\|^2\right)\cr
=   &\frac{\eta L_r}{\sqrt{m}}\Big(\EE\|\nabla f(\vx^t)\|^2+\sum\limits_{r=0}^{t-1}q_r\EE\|\nabla f(\vx^{t-r})\|^2+c_t\|\nabla f(\vx^0)\|^2\Big)
\end{align}
and
\begin{align}\label{eq:rel2}
    &\EE\big[\|\nabla f(\vx^{t+1})\|^2-\|\nabla f(\vx^t)\|^2\big]\cr
\le &\EE\big[\|\nabla f(\vx^{t+1})+\nabla f(\vx^t)\|\cdot\|\nabla f(\vx^{t+1})-\nabla f(\vx^t)\|\big]\cr
\le &L_r\EE\left[\big(2\|\nabla f(\vx^t)\|+\|\nabla f(\vx^{t+1})-\nabla f(\vx^t)\|\big)\|\vx^{t+1}- \vx^t\|\right]\cr
\le &L_r\EE\left[2\|\nabla f(\vx^t)\|\cdot\|\vx^{t+1}- \vx^t\|+L_r\|\vx^{t+1}- \vx^t\|^2\right]\cr
=   &L_r\EE\left[2\eta\|\nabla f(\vx^t)\|\cdot\|U_{i_t}\nabla f(\vx^{t-j_t})\|+\eta^2 L_r\|U_{i_t}\nabla f(\vx^{t-j_t})\|^2\right]\cr
\le &L_r\EE\left[\frac{\eta}{\sqrt{m}}\|\nabla f(\vx^t)\|^2+\eta\sqrt{m}\|U_{i_t}\nabla f(\vx^{t-j_t})\|^2+\eta^2 L_r\|U_{i_t}\nabla f(\vx^{t-j_t})\|^2\right]\cr
=   &\frac{\eta L_r}{\sqrt{m}}\EE\|\nabla f(\vx^t)\|^2+\left(\frac{\eta L_r}{\sqrt{m}}+\frac{\eta^2 L_r^2}{m}\right)\EE\|\nabla f(\vx^{t-j_t})\|^2\cr
=   &\frac{\eta L_r}{\sqrt{m}}\EE\|\nabla f(\vx^t)\|^2+\left(\frac{\eta L_r}{\sqrt{m}}+\frac{\eta^2 L_r^2}{m}\right)\left(\sum\limits_{r=0}^{t-1}q_r\EE\|\nabla f(\vx^{t-r})\|^2+c_t\|\nabla f(\vx^0)\|^2\right).
\end{align}
We first show the first inequality in~\eqref{eq:bdgrad}. Note that~\eqref{sm-eta-rate} gives us
\begin{equation}\label{sm-eta-rate2}
\textstyle \frac{1}{1-(1+M_\rho)\frac{\eta L_r}{\sqrt{m}}}\le \rho.
\end{equation}
When $t=0$, we have from~\eqref{eq:rel1} that
$\textstyle \|\nabla f(\vx^0)\|^2-\EE\|\nabla f(\vx^1)\|^2\le \frac{2 \eta L_r}{\sqrt{m}}\|\nabla f(\vx^0)\|^2\le (1+M_\rho)\frac{\eta L_r}{\sqrt{m}}\|\nabla f(\vx^0)\|^2.$
Hence, $\|\nabla f(\vx^0)\|^2\le \rho \EE\|\nabla f(\vx^1)\|^2$ from~\eqref{sm-eta-rate2}. Now we assume that $\EE\|\nabla f(\vx^t)\|^2\le \rho\EE\|\nabla f(\vx^{t+1})\|^2$ for all $t\le k-1$. For $t=k$, it holds from~\eqref{eq:rel1} and the induction assumption that
\begin{align*}
    &\EE\|\nabla f(\vx^k)\|^2-\EE\|\nabla f(\vx^{k+1})\|^2\cr
\le &\frac{\eta L_r}{\sqrt{m}}\left(\EE\|\nabla f(\vx^k)\|^2+\sum\limits_{t=0}^{k-1}q_t\rho^t\EE\|\nabla f(\vx^k)\|^2+c_k\rho^k\EE\|\nabla f(\vx^k)\|^2\right)\\
=   &\frac{\eta L_r}{\sqrt{m}}\left(1+\sum\limits_{t=0}^{k-1}q_t\rho^t+c_k\rho^k\right)\EE\|\nabla f(\vx^k)\|^2
\le \frac{\eta L_r}{\sqrt{m}}(1+M_\rho)\cdot \EE\|\nabla f(\vx^k)\|^2.
\end{align*}
Hence, we have $\EE\|\nabla f(\vx^k)\|^2\le \rho\EE\|\nabla f(\vx^{k+1})\|^2$ from~\eqref{sm-eta-rate2}. Therefore, we finish the induction step, and thus the first inequality of~\eqref{eq:bdgrad} holds.

Next we show the second inequality of~\eqref{eq:bdgrad}. Since~\eqref{sm-eta-rate} implies $\textstyle \eta\le \frac{\rho-1}{\frac{L_r}{\sqrt{m}}\left(1+M_\rho+\frac{(\rho-1)M_\rho}{\rho(1+M_\rho)}\right)}$, 
\begin{align}\label{sm-eta-rate3}
1+\frac{\eta L_r}{\sqrt{m}}+\left(\frac{\eta L_r}{\sqrt{m}}+\frac{\eta^2 L_r^2}{m}\right)M_\rho\overset{(\ref{sm-eta-rate})}
\le &1+\frac{\eta L_r}{\sqrt{m}}(1+M_\rho)+M_\rho \frac{\eta L_r^2}{m}\frac{(\rho-1)\sqrt{m}}{\rho L_r(1+M_\rho)}\cr
=   &1+\frac{\eta L_r}{\sqrt{m}}\left(1+M_\rho+\frac{(\rho-1)M_\rho}{\rho(1+M_\rho)}\right)\le\rho.
\end{align}
When $t=0$, we  have from~\eqref{eq:rel2} that
\begin{align*}
\EE\|\nabla f(\vx^1)\|^2-\|\nabla f(\vx^0)\|^2
\le &\big(\frac{2\eta L_r}{\sqrt{m}}+\frac{\eta^2 L_r^2}{m}\big)\|\nabla f(\vx^0)\|^2\\
\le &\left((1+M_\rho)\frac{\eta L_r}{\sqrt{m}}+\frac{\eta^2 L_r^2}{m}\right)\|\nabla f(\vx^0)\|^2.
\end{align*}
Hence, $\EE\|\nabla f(\vx^1)\|^2\le \rho \|\nabla f(\vx^0)\|^2$ holds from~\eqref{sm-eta-rate3}. Assume $\EE\|\nabla f(\vx^{t+1})\|^2\le \rho\EE\|\nabla f(\vx^t)\|^2$ for all $t\le k-1$. It follows from~\eqref{eq:rel2} and the induction assumption that
\begin{align*}
    &\EE\|\nabla f(\vx^{k+1})\|^2-\EE\|\nabla f(\vx^{k})\|^2\\
\le &\frac{\eta L_r}{\sqrt{m}} \EE\|\nabla f(\vx^{k})\|^2 +\left(\frac{\eta L_r}{\sqrt{m}}+\frac{\eta^2 L_r^2}{m}\right)\left(\sum\limits_{t=0}^{k-1}q_t\rho^t\EE\|\nabla f(\vx^k)\|^2+c_k\rho^k\EE\|\nabla f(\vx^k)\|^2\right)\\
=   &\left(\frac{\eta L_r}{\sqrt{m}}+\left(\frac{\eta L_r}{\sqrt{m}}+\frac{\eta^2 L_r^2}{m}\right)\left(\sum\limits_{t=0}^{k-1}q_t\rho^t+c_k\rho^k\right)\right)\EE\|\nabla f(\vx^{k})\|^2\\
\le &\left(\frac{\eta L_r}{\sqrt{m}}+\left(\frac{\eta L_r}{\sqrt{m}}+\frac{\eta^2 L_r^2}{m}\right)M_\rho\right)\EE\|\nabla f(\vx^{k})\|^2.
\end{align*}
Hence, from~\eqref{sm-eta-rate3}, $\EE\|\nabla f(\vx^{k+1})\|^2\le \rho\EE\|\nabla f(\vx^k)\|^2$ holds, and we complete the proof.
\qed\end{proof}









%
%
%

\end{document}

%% file: macros.tex
\usepackage{mathtools}

\usepackage[normalem]{ulem} 

\newcommand{\bm}[1]{\boldsymbol{#1}}

\newcommand{\vb}{{\mathbf{b}}}

\newcommand{\vd}{{\mathbf{d}}}

\newcommand{\vj}{{\mathbf{j}}}

\newcommand{\vu}{{\mathbf{u}}}
\newcommand{\vv}{{\mathbf{v}}}

\newcommand{\vx}{{\mathbf{x}}}
\newcommand{\vy}{{\mathbf{y}}}

\newcommand{\vA}{{\mathbf{A}}}

\newcommand{\vX}{{\mathbf{X}}}
\newcommand{\vY}{{\mathbf{Y}}}
\newcommand{\vZ}{{\mathbf{Z}}}

\newcommand{\cK}{{\mathcal{K}}}

\newcommand{\cP}{{\mathcal{P}}}

\newcommand{\EE}{\mathbb{E}} 
\newcommand{\RR}{\mathbb{R}} 
\newcommand{\vzero}{\mathbf{0}} 

\newcommand{\Prob}{{\mathrm{Prob}}} 
\newcommand{\prox}{{\mathbf{prox}}} 

\newcommand{\st}{\mbox{ s.t. }}

\DeclareMathOperator*{\argmin}{arg\,min} 

\DeclareMathOperator*{\Min}{minimize}

\newcommand{\bc}{\begin{center}}
\newcommand{\ec}{\end{center}}

\newcommand{\bdm}{\begin{displaymath}}
\newcommand{\edm}{\end{displaymath}}

\newcommand{\beq}{\begin{equation}}
\newcommand{\eeq}{\end{equation}}

\newcommand{\bfl}{\begin{flushleft}}
\newcommand{\efl}{\end{flushleft}}

\newcommand{\bt}{\begin{tabbing}}
\newcommand{\et}{\end{tabbing}}

\newcommand{\beqn}{\begin{eqnarray}}
\newcommand{\eeqn}{\end{eqnarray}}

\newcommand{\beqs}{\begin{align*}} 
\newcommand{\eeqs}{\end{align*}}  

\newtheorem{assumption}{Assumption}

\newcommand{\hvx}{{\hat{\vx}}}